%% file: Energy_decay_for_the_damped_wave_equation_with_a_confining_potential_in_the_Euclidean_space.tex
\documentclass[english,10pt,a4paper]{amsart}
\usepackage[english]{babel}
\usepackage{amsmath,amssymb,amsthm,amsfonts}
\usepackage{graphicx}
\usepackage[left=2.8cm,right=2.8cm,top=2.8cm,bottom=2.8cm]{geometry}

\usepackage{empheq}
\usepackage{enumitem}
\usepackage{lmodern}
\usepackage{esint}
\usepackage{braket}
\usepackage[mathscr]{eucal}
\usepackage{dsfont}

\usepackage[colorlinks=true]{hyperref}
\hypersetup{citecolor=blue}

\usepackage{caption}
\usepackage{subcaption}

\input macros_Antoine.tex

\newcommand*{\Bop}{\scr{B}}
\newcommand*{\dwop}{\scr{A}} 
\newcommand*{\hilbert}{\frak{H}}

\DeclarePairedDelimiter\avg{\langle}{\rangle}

\newcommand*{\sympf}{\sigma}

\numberwithin{equation}{section}

\DeclareRobustCommand{\SkipTocEntry}[5]{}

\begin{document}

\title[Uniform stability of the damped wave equation with a confining potential]{Uniform stability of the damped wave equation\\with a confining potential in the Euclidean space}
\date{\today}
\author{Antoine Prouff}
\address{Université Paris-Saclay, CNRS, Laboratoire de Mathématiques d'Orsay, 91405, Orsay, France.}
\email{antoine.prouff@universite-paris-saclay.fr}

\begin{abstract}
We investigate trend to equilibrium for the damped wave equation with a confining potential in the Euclidean space. We provide with necessary and sufficient geometric conditions for the energy to decay exponentially uniformly. The proofs rely on tools from semiclassical analysis together with the construction of quasimodes of the damped wave operator. In addition to the Geometric Control Condition, which is familiar in the context of compact Riemannian manifolds, our work involves a new geometric condition due to the presence of turning points in the underlying classical dynamics which rules the propagation of waves in the high-energy asymptotics.
\end{abstract}

\maketitle
{
\hypersetup{linkcolor=black}
\tableofcontents
}

\section{Introduction}

\subsection{Setting of the problem}

We study the energy decay of damped waves in the Euclidean space of dimension $d \ge 1$, that is to say the trend to equilibrium for solutions to the equation
\begin{equation} \label{eq:eq}
\left\{
\begin{aligned}
\partial_t^2 u + P u + b(x) \partial_t u &= 0 , \qquad x \in \R^d, t > 0 , \\
(u, \partial_t u)_{\vert t = 0} &= U_0 = (u_0, u_1) .
\end{aligned}
\right.
\end{equation}
In~\eqref{eq:eq}, $b \in L^\infty(\R^d)$ is the non-negative \emph{damping coefficient} and $P$ is defined by
\begin{equation} \label{eq:defP}
P = V(x) - \tfrac{1}{2} \Lap ,
\end{equation}
where $V$ is a non-negative locally bounded confining potential on $\R^d$:
\begin{equation} \label{eq:assumV}
V \in L_\loc^\infty(\R^d; \R),
	\qquad
V \ge 0
	\qquad \rm{and} \qquad
V(x) \strongto{x \to \infty} + \infty .
\end{equation}
A further growth condition on $V$ will be stated later. As a consequence of~\eqref{eq:assumV}, the operator $P$, with domain
\begin{equation*}
\dom P
	= \set{u \in L^2(\R^d)}{\left(V(x) - \tfrac{1}{2} \Delta\right) u \in L^2(\R^d)} ,
\end{equation*}
is self-adjoint and positive definite. The damped wave equation~\eqref{eq:eq} is a well-posed evolution problem~\cite[Theorem 4.3]{Pazy:83} on the Hilbert space $\hilbert := \dom P^{1/2} \oplus L^2(\R^d)$, whose inner product is given by
\begin{equation*}
\inp*{(u_1, v_1)}{(u_2, v_2)}_{\frak{H}}
	:= \inp*{P^{1/2} u_1}{P^{1/2} u_2}_{L^2(\R^d)} + \inp*{v_1}{v_2}_{L^2(\R^d)} ,
		\qquad (u_1, v_1), (u_2, v_2) \in \hilbert .
\end{equation*}

As a first observation, we mention that the energy balance corresponding to this equation, obtained formally by multiplying~\eqref{eq:eq} by $\partial_t \bar u$ and performing integration by parts, is simply
\begin{equation} \label{eq:energybalance}
\dfrac{\dd}{\dd t} \cal{E}\left(U_0, t\right)
	= - \int_{\R^d} b \abs*{\partial_t u(t)}^2 \dd x ,
\end{equation}
where the energy $\cal{E}$ is defined by
\begin{equation*}
\cal{E}\left(U_0, t\right)
	:= \dfrac{1}{2} \left( \norm*{P^{1/2} u(t)}_{L^2(\R^d)}^2 + \norm*{\partial_t u(t)}_{L^2(\R^d)}^2 \right) .
\end{equation*}
This a priori estimate indicates that the equation~\eqref{eq:eq} is non-conservative wherever $b$ is positive. In the present work, we investigate necessary and sufficient conditions under which uniform exponential decay for Equation~\eqref{eq:eq} holds, that is to say there exist constants $C > 0, \tau > 0$ such that
\begin{equation} \label{eq:decayenergy}
\forall U_0 \in \hilbert, \quad
	\cal{E}\left(U_0, t\right)
		\le C \e^{-t/\tau} \cal{E}\left(U_0, 0\right) ,
			\qquad \forall t \ge 0 .
\end{equation}

The study of decay rates for the damped wave equation dates back to the 70s with the celebrated works of Rauch and Taylor, and later of Bardos, Lebeau and Rauch~\cite{RT:74,BLR:88,BLR:92}, in the setting of compact Riemannian manifolds. They prove a sharp sufficient condition for having uniform exponential decay of the energy of solutions to the damped wave equation. The latter can be expressed in broad terms as follows: any geodesic enters the so-called damped set, where the damping is effective (this is merely $\{b > 0\}$ when $b$ is continuous). This is called the \emph{Geometric Control Condition} (GCC). The idea behind this result is that the energy of high frequency solutions to wave equations is largely carried by the rays of geometric optics.

In comparison, few is understood in unbounded geometries, partly because it is unclear how to handle properly the presence of infinity in space. The paper of Burq and Joly~\cite{BJ:16} provides a sufficient condition for having uniform decay of the energy for the Klein--Gordon equation in $\R^d$, that is to say the same equation as~\eqref{eq:eq} with a bounded potential and a Laplacian with possibly varying coefficients. For the flat Laplacian case with constant potential ($P = - \Delta + 1$), it reads as follows: uniform exponential decay~\eqref{eq:decayenergy} holds if there exists $L > 0$ such that the average of $b$ on any segment of length $L$ is uniformly bounded from below by a constant $c > 0$. This uniform version of the GCC appears to be a fairly natural generalization of the GCC to unbounded domains. Note that their proof requires $b$ to be uniformly continuous on $\R^d$, which in particular prevents $b$ from being more and more oscillatory near infinity. Their result has been recently generalized to asymptotically cylindrical and conic manifolds by Wang in~\cite{W:20}. For other works on the damped wave equation in the Euclidean space, we refer to the papers by Bouclet--Royer~\cite{BR:14} and Royer~\cite{Royer:18energyspace} for a study of the local energy decay with a short-range damping coefficient. See also investigations of the effect of periodic damping coefficients by Wunsch~\cite{Wunsch:17}, Joly--Royer~\cite{JR:18}, Royer~\cite{Royer:18highlyoscillating}.

In the sequel, we may rewrite~\eqref{eq:eq} in the form of an order-one evolution PDE:
\begin{equation} \label{eq:eq1}
\partial_t U
	= \dwop_b U , \qquad t > 0 ,
\end{equation}
where we wrote $U = (u, \partial_t u)$, and the infinitesimal generator is the so-called \emph{damped wave operator}:
\begin{equation*} \label{eq:IG}
\dwop_b =
\begin{pmatrix}
0 & 1 \\
- P & - b
\end{pmatrix} ,
\end{equation*}
acting on $\hilbert$, with domain $\dom \dwop_b = \dom P \oplus \dom P^{1/2}$.
The energy of the solution of~\eqref{eq:eq1} with initial datum $U_0$ corresponds to:
\begin{equation*} \label{eq:energy}
\cal{E}\left(U_0, t\right)
	= \dfrac{1}{2} \norm*{\e^{t \dwop_b} U_0}_\hilbert^2 ,
		\qquad t \in \R_+ .
\end{equation*}

The exponential decay, defined in~\eqref{eq:decayenergy}, can be expressed in a simple way by saying that the norm of the semigroup decays exponentially over time.

\begin{definition}[Uniform stabilization] \label{def:unifstab}
The equation~\eqref{eq:eq} is said to be uniformly stable if there exist $C > 0, \tau > 0$ such that
\begin{equation*}
\norm*{\e^{t \dwop_b}}_{\Bop(\hilbert)} \le C \e^{-t/\tau} , \qquad \forall t \ge 0 .
\end{equation*}
\end{definition}

Note that by the semigroup property, uniform stability is equivalent to the fact that $\norm*{\e^{t \dwop_b}}_{\Bop(\hilbert)} < 1$ for some $t > 0$ (see~\cite[Chapter V, Proposition 1.7]{EN:book}).

A motivation to study~\eqref{eq:eq} in the setting~\eqref{eq:defP} and~\eqref{eq:assumV} is that waves should be essentially trapped in a bounded region of $\R^d$ when $V$ is confining. Then we can expect the analysis to share similarities with that of damped waves in compact domains. Actually, the situation will turn out to be rather different, due to the fact that the underlying classical dynamics possesses turning points, as we shall explain further below.

\subsection{Main results}

We present two types of results here. We refer to the first one as ``a priori conditions of stabilization" since they do not rely on any particular feature of the potential (except for~\eqref{eq:assumV}). These results are based on geometrical optics constructions and semiclassical defect measures for the flat Laplacian in $\R^d$.  Thus Propositions~\ref{prop:aprioriNC} and~\ref{prop:aprioriSC} below may be seen as a variation of classical results in control theory.

Throughout this paper, we denote by $\abs{A}$ the {$d$-dimensional} Lebesgue measure of the measurable set $A \subset \R^d$. Given a finite measure space $(X, \cal{B}, \mu)$, we also use the notation
\begin{equation*}
\fint_X f \dd \mu
	:= \dfrac{1}{\mu(X)} \int_X f \dd \mu .
\end{equation*}
We also write $S \R^d = \R^d \times \sph^{d-1}$.

\begin{proposition}[A priori necessary condition for uniform stability] \label{prop:aprioriNC}
Assume $V$ is subject to~\eqref{eq:assumV} and $b \in L^\infty(\R^d)$ is a non-negative damping coefficient. Then if~\eqref{eq:eq} is uniformly stable, $b$ satisfies the Uniform Geometric Control Condition
\begin{equation} \label{eq:GCC}
\exists T > 0, \exists c > 0 : \forall (x_0, \nu_0) \in S \R^d, \forall r > 0, \qquad \fint_{-T}^T (b \ast \kappa_r)(x_0 + t \nu_0) \dd t \ge c ,
\tag{UGCC}
\end{equation}
where
\begin{equation} \label{eq:defkappa}
\kappa := \dfrac{1}{\abs{B_1(0)}} \one_{B_1(0)} ,
	\qquad
\kappa_r := r^{-d} \kappa\left(\dfrac{\bullet}{r}\right), \;\, r > 0 .
\end{equation}
\end{proposition}

\begin{remark} \label{rmk:GCCnear0}
We prove in Lemma~\ref{lem:GCCnear0} of Appendix~\ref{app:constant} that the condition~\eqref{eq:GCC} is equivalent to the following statement:
\begin{equation} \label{eq:GCCnear0}
\exists T > 0, \exists c > 0 : \forall (x_0, \nu_0) \in S \R^d, \qquad \liminf_{r \to 0} \fint_{-T}^T (b \ast \kappa_r)(x_0 + t \nu_0) \dd t \ge c ,
\end{equation}
where $\kappa_r$ is defined in~\eqref{eq:defkappa}.
\end{remark}

\begin{proposition}[A priori sufficient condition for uniform stability] \label{prop:aprioriSC}
Assume $V$ is subject to~\eqref{eq:assumV} and $b \in L^\infty(\R^d)$ is a non-negative damping coefficient. If $b$ is essentially bounded from below by a positive constant outside some compact set, that is
\begin{equation} \label{eq:boundedoutsidecompactset}
\exists R > 0, \exists c > 0 : \qquad b(x) \ge c \quad \rm{for \; a.e.} \; x \in \R^d \setminus B_R(0) ,
\end{equation}
then~\eqref{eq:eq} is uniformly stable.
\end{proposition}

The version of the Geometric Control Condition appearing in Proposition~\ref{prop:aprioriNC} generalizes the one known for continuous damping coefficients (see~\cite{RT:74,BLR:92}) to the case where $b$ is merely $L^\infty$. It has already been formulated (in an equivalent way) by Burq and Gérard in~\cite{BG:20}, and has been shown to be necessary for uniform stability to occur for the damped wave equation on compact Riemannian manifolds. Notice that for $b \in \cont^0(\R^d)$, the condition~\eqref{eq:GCC} is equivalent to
\begin{equation*}
\exists T > 0, \exists c > 0 : \forall (x_0, \nu_0) \in S \R^d, \qquad \fint_{-T}^T b(x_0 + t \nu_0) \dd t \ge c .
\end{equation*}
It is also the natural way of generalizing GCC to non-compact geometries---see Burq and Joly~\cite{BJ:16}. As for the sufficient condition of Proposition~\ref{prop:aprioriSC}, it already appeared in~\cite{Zuazua:91} in the context of semilinear damped wave equations in $\R^d$.
Although we assume in Propositions~\ref{prop:aprioriNC} and~\ref{prop:aprioriSC} that $V$ is a confining potential, the conditions~\eqref{eq:GCC} and~\eqref{eq:boundedoutsidecompactset} are unrelated to this confinement property, so that Propositions~\ref{prop:aprioriNC} and~\ref{prop:aprioriSC} are actually true also for non-negative $L_\loc^\infty(\R^d)$ potentials.\footnote{One may assume that this potential is such that $P = V(x) - \tfrac{1}{2} \Delta$ is positive definite.} The assumption that $V$ is confining is convenient when reducing the problem of uniform stability to a resolvent estimate (Proposition~\ref{prop:quasimodes}).

Our second type of results aims at understanding the effect of the presence of a \emph{confining} potential. A new geometric control condition appears in this context, under additional assumptions on the potential.

\begin{assumption}[Growth condition on the potential] \label{assum:assumptionspotentialdampedeq}
The potential $V$ is of class $\cont^2$ and strictly sub-quartic, that is to say
\begin{equation*}
V \in \cont^2(\R^d)
	\qquad \rm{and} \qquad
\nabla V(x) = o\left(V(x)^{3/4}\right)
	\quad \rm{as} \quad x \to \infty .
\end{equation*}
\end{assumption}

\begin{remark}
One can check that the strict sub-quarticity condition implies the (weaker) property $V(x) \ll \abs*{x}^4$ as $x \to \infty$. If fact, every $C^2$ potential of the form $\abs*{x}^s$ with $0 < s < 4$ outside a compact set satisfies the strict sub-quarticity assumption.
\end{remark}

Our main result is the following necessary condition.

\begin{theorem}[Necessary condition for uniform stability] \label{thm:necessarycondition}
Assume the potential $V$ is subject to~\eqref{eq:assumV} and Assumption~\ref{assum:assumptionspotentialdampedeq}, and the damping coefficient $b \in L^\infty(\R^d)$ is such that the damped wave equation~\eqref{eq:eq} is uniformly stable. Then $b$ satisfies the Uniform Geometric Control Condition~\eqref{eq:GCC} and the Turning Point Condition:
\begin{equation} \label{eq:TPC}
\tag{TPC}
\exists R > 0 : \qquad \liminf_{x \to \infty} \fint_{B_{R/V(x)^{1/4}}(x)} b(y) \dd y > 0 .
\end{equation}
\end{theorem}

\begin{figure}
\centering
\hspace*{-1cm}
\begin{subfigure}{0.40\textwidth}
         \centering
         \includegraphics[scale=0.18]{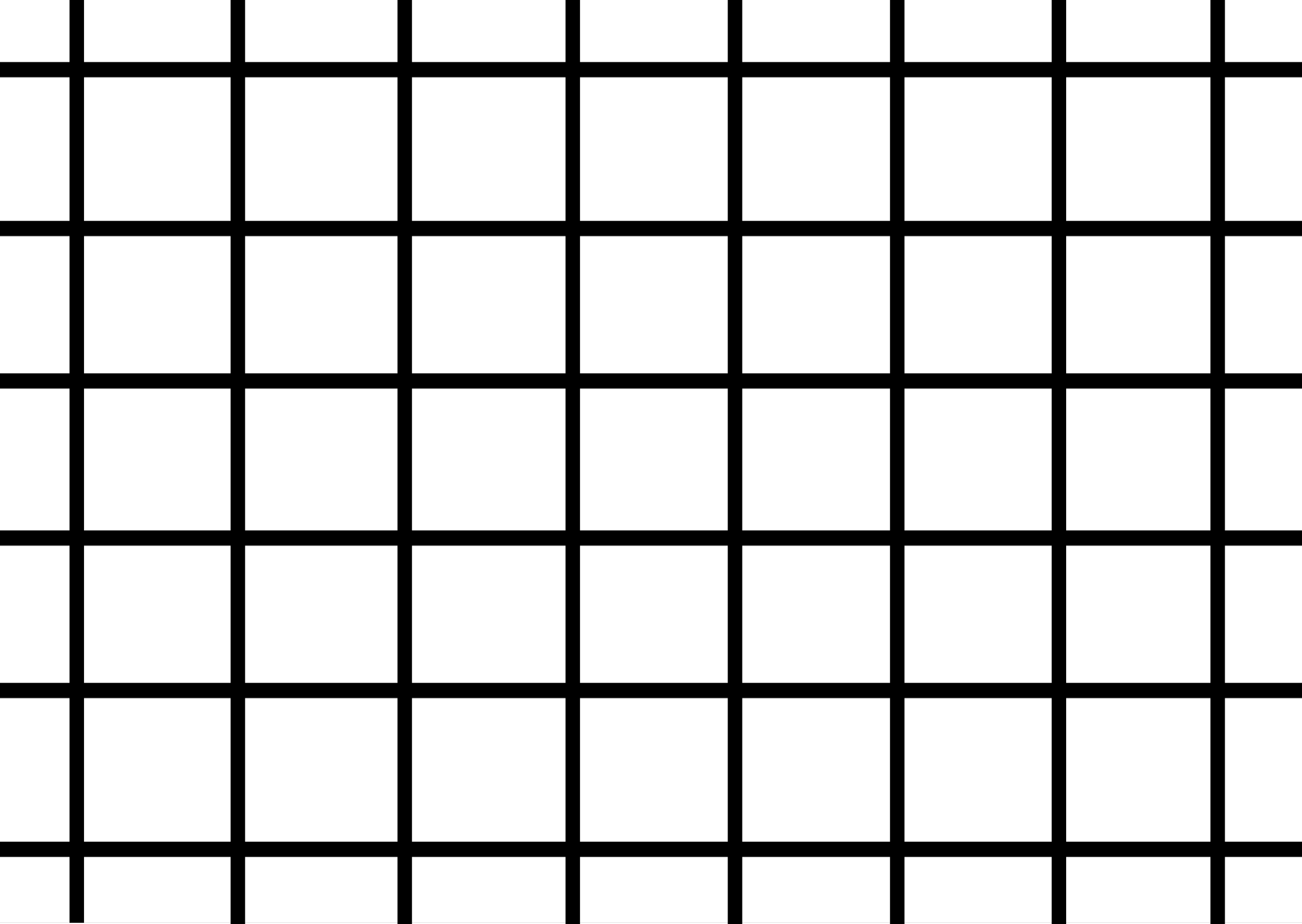}
         \caption{Damping coefficient satisfying~\eqref{eq:GCC}.}
         \label{subfig:one}
     \end{subfigure}
\hspace*{1cm}
     \begin{subfigure}{0.40\textwidth}
         \centering
         \includegraphics[scale=0.18]{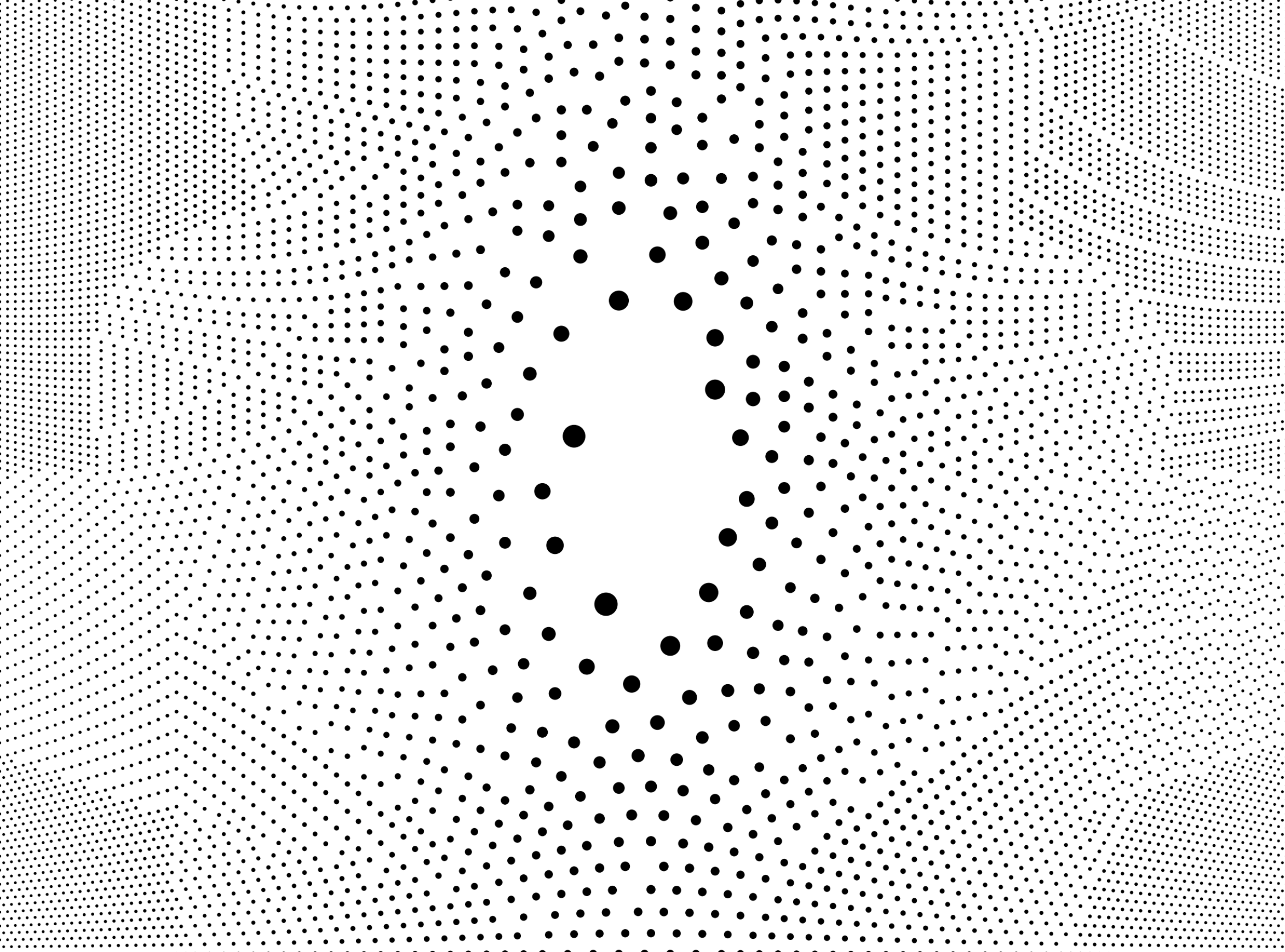}
         \caption{Damping coefficient satisfying~\eqref{eq:TPC}.}
         \label{subfig:deco}
     \end{subfigure}
        \caption{Black and white indicate the points where $b$ is equal to $1$ or $0$ respectively.}
        \label{fig:amortissement}
\end{figure}

Restricting the scope of our investigation to a class of nice damping coefficients, we prove the following characterization of uniform stabilization, as a consequence of the necessary condition in Theorem~\ref{thm:necessarycondition} and the a priori sufficient condition obtained in Proposition~\ref{prop:aprioriSC}.

\begin{corollary}[Characterization of uniform stability for uniformly continuous damping coefficients] \label{cor:characterizationregulardamping}
Let $V$ be a potential on $\R^d$ satisfying~\eqref{eq:assumV} and Assumption~\ref{assum:assumptionspotentialdampedeq}, and suppose $b$ is a \emph{uniformly continuous} non-negative damping coefficient. Then equation~\eqref{eq:eq} is uniformly stable if and only if $b$ is bounded from below outside a compact set, that is to say
\begin{equation*}
\liminf_{x \to \infty} b(x) > 0 .
\end{equation*}
\end{corollary}

The point of Corollary~\ref{cor:characterizationregulardamping} is to show that the necessary condition~\eqref{eq:TPC} together with uniform continuity of $b$ forces the latter to be bounded from below in a neighborhood of infinity.

\subsection{Geometric and dynamical stabilization conditions}

Proposition~\ref{prop:aprioriNC} and Theorem~\ref{thm:necessarycondition} above provide geometric necessary conditions of stabilization. To understand completely the picture, it is relevant to unify the two conditions~\eqref{eq:GCC} and~\eqref{eq:TPC} into a single dynamical condition, accounting for the underlying classical dynamics (here a distorted geometric optics).

\subsubsection{Classical dynamics}

We briefly introduce a Hamiltonian flow associated with $P$ and recall the basic definitions that will be needed thereafter. Throughout the article, a typical point of $\R^d$ will be denoted by $x$ or $y$, whereas the points of the cotangent space $T_x^\star \R^d$ will be denoted by $\xi$ or $\eta$. We will also often use the notation $\rho$ or $\zeta$ to denote a running point $(x, \xi)$ or $(y, \eta)$ in $T^\star \R^d$. The cotangent space $T^\star \R^d$ equipped with the symplectic form $\sympf = \dd \xi \wedge \dd x$, the so-called \emph{phase space}, will be frequently identified with $\R^{2d}$.
The operator $P$ can be written as $P = \Opw{p}$, that is to say the Weyl quantization of the real-valued symbol $p(x, \xi) = V(x) + \frac{1}{2} \abs*{\xi}^2$ (see~\cite{Hoermander:V3,Lerner:10,Zworski:book}). Under Assumption~\ref{assum:assumptionspotentialdampedeq}, the potential $V$ is $C^2$, so we can define the Hamiltonian vector field associated with $p$ by
\begin{equation*}
H_p
	= \xi \cdot \partial_x - \nabla V(x) \cdot \partial_\xi .
\end{equation*}
The Hamiltonian flow\footnote{In fact, $V \in W_\loc^{2, \infty}(\R^d)$ would be sufficient to define the Hamiltonian flow via the Cauchy--Lipschitz theorem.} associated with $p$, which we denote by $(\varphi^t)_{t \in \R}$, is the flow generated by $H_p$, namely
\begin{equation*}
\dfrac{\dd}{\dd t} \varphi^t(\rho) = H_p\left(\varphi^t(\rho)\right) .
\end{equation*}
Writing $\varphi^t(\rho) = (x^t, \xi^t)$ the position and momentum components of the flow, we can rewrite this as
\begin{equation} \label{eq:Hamilton}
\begin{split}
\left\{
\begin{aligned}
\dot x^t
	&= \xi^t \\
\dot \xi^t
	&= - \nabla V(x^t)
\end{aligned}
\right. ,
\end{split}
\end{equation}
which corresponds to Newton's second law of classical mechanics with a force field $\vec F = - \nabla V$. In the setting of a confining potential~\eqref{eq:assumV}, which is our case throughout this paper, classical o.d.e.\ theory ensures that the integral curves of $H_p$ are well-defined for all times (the level sets of $p$ are compact). Notice that $p$ is preserved by the Hamiltonian flow by construction.

\subsubsection{Dynamical stabilization condition}

Given a continuous function $K$ on $T^\star \R^d$ and a real number $T> 0$, we denote by $\avg*{K}_T$ the function on $T^\star \R^d$ consisting in averaging $K$ along the trajectories of the Hamiltonian flow, namely
\begin{equation*}
\avg*{K}_T(\rho)
	:= \dfrac{1}{2T} \int_{-T}^T K\left( \varphi^t(\rho) \right) \dd t ,
		\qquad \rho \in T^\star \R^d .
\end{equation*}
If $K$ is a continuous function on $\R^d$, we write $\avg*{K}_T(\rho) := \avg*{K \circ \pi}_T(\rho)$ (with a slight abuse of notation), where $\pi : T^\star \R^d \to \R^d$ is the cotangent bundle projection $\pi(x, \xi) = x$.

\begin{definition}[Stabilization condition \--- Dynamical formulation] \label{def:stabconddyn}
Let $V$ be a $\cont^2$ potential such that the Hamiltonian flow $(\varphi^t)_{t \in \R}$ associated with $p(x, \xi) = V(x) + \frac{1}{2} \abs*{\xi}^2$ introduced in~\eqref{eq:Hamilton} is globally well-defined. We say a function $b \in L^\infty(\R^d)$ satisfies the \emph{Dynamical Stabilization Condition} if
\begin{equation} \label{eq:SCdyn}
\tag{DSC}
\exists T > 0, \exists R > 0 : \qquad
	\liminf_{\lambda \to + \infty} \inf_{\rho \in \{p = \lambda^2\}} \avg*{b \ast \kappa_{R/\sqrt{\lambda}}}_{T/\lambda}(\rho)  > 0 .
\end{equation}
\end{definition}

We refer to Subsection~\ref{subsec:scales} below for a discussion of the various scales appearing in this definition, namely the energy scale $\{p = \lambda^2\}$, the space-regularization scale $R/\sqrt{\lambda}$ and the time scale $T/\lambda$.

\begin{proposition} \label{prop:SCgeom}
Suppose $V$ is subject to~\eqref{eq:assumV} and Assumption~\ref{assum:assumptionspotentialdampedeq}. Then a non-negative function $b \in L^\infty(\R^d)$ satisfies the Dynamical Stabilization Condition~\eqref{eq:SCdyn} if and only if it satisfies the Uniform Geometric Control Condition~\eqref{eq:GCC} and the Turning Point Condition~\eqref{eq:TPC}:
\begin{equation*}
\eqref{eq:SCdyn}
	\quad \Longleftrightarrow \quad
\eqref{eq:GCC} + \eqref{eq:TPC} .
\end{equation*}
\end{proposition}

The Dynamical Stabilization Condition~\eqref{eq:SCdyn} is the natural candidate for characterizing the uniform stabilization of the damped wave equation~\eqref{eq:eq}. Indeed, it follows directly from the quantum-classical correspondence heuristics that the damping coefficient should capture all the high energy rays of geometric optics (modulo the non-obvious space-time scales $R/\sqrt{\lambda}$ and $T/\lambda$ described below in Subsection~\ref{subsec:scales}). The geometrical conditions~\eqref{eq:GCC} and~\eqref{eq:TPC} are more tractable in practice, so Proposition~\ref{prop:SCgeom} above gives us a very practical way to check that~\eqref{eq:SCdyn} is verified on concrete examples.

The Geometric Control Condition and the Turning Point Condition are related to the Dynamical Stabilization Condition in distinct phase space regions. Condition~\eqref{eq:GCC} concerns the kinetic regime where the Laplace operator rules over the potential. It is obtained by considering a fixed point $x_0 \in \R^d$ and looking at points over $x_0$ in the energy shell $\{p = \lambda^2\}$ with $\lambda \gg 1$. In a $O(1)$ neighborhood of $x_0$ on the energy layer $\{p= \lambda^2\}$, the Hamiltonian $p$ is well-approximated by $\abs{\xi}^2$, so that the flow lines look more and more like straight lines and project to rays onto the physical space (this corresponds to taking $x$ fixed and letting $\xi \to \infty$ on $V(x) + \frac{1}{2} \abs{\xi}^2 = \lambda^2$). In contrast, the condition~\eqref{eq:TPC} comes from the potential regime, where states have a large potential energy compared to their oscillation rate, i.e.\ $V$ prevails over the Laplacian (this corresponds to taking $\xi$ fixed and letting $x \to \infty$ on $V(x) + \frac{1}{2} \abs{\xi}^2 = \lambda^2$). This regime arises in phase space regions close to the null section but also close to infinity in space, where the classical dynamics has turning points: trajectories are transversal to the null section and project singularly onto the physical space.
Proposition~\ref{prop:SCgeom} can be understood as follows: controlling only the kinetic regime through~\eqref{eq:GCC} and the potential regime through~\eqref{eq:TPC} is sufficient to control any state on the energy shell $\{p = \lambda^2\}$.

Proposition~\ref{prop:SCgeom} suggests that the condition $\eqref{eq:GCC} + \eqref{eq:TPC}$ is sharp. This motivates the following statement.

\begin{conjecture}
Consider $V$ satisfying~\eqref{eq:assumV} and Assumption~\ref{assum:assumptionspotentialdampedeq}, and $b \in L^\infty(\R^d)$ non-negative. Then Equation~\eqref{eq:eq} is uniformly stable (in the sense of Definition~\ref{def:unifstab}) if and only if $b$ satisfies the Geometric Control Condition~\eqref{eq:GCC} and the Turning Point Condition~\eqref{eq:TPC}.
\end{conjecture}

The dynamical viewpoint of~\eqref{eq:SCdyn} seems to be the best one to quantify the decay rate, namely the optimal rate $\tau$ in~\eqref{eq:decayenergy} (at least for high energy initial data), as Lebeau did in compact manifolds~\cite{Leb:96}. We define in Appendix~\ref{app:constant} a quantity which we guess could be related to this optimal decay rate.

\subsubsection{Typical space-time scales} \label{subsec:scales}

We explain the typical scales arising in~\eqref{eq:SCdyn}. In view of the quantum-classical correspondence principle (or Egorov's theorem, see~\cite[Chapter 11]{Zworski:book} or~\cite[Chapter 7, Section 8]{Taylor:vol2}), one can guess that a sharp geometric condition on $b$ to ensure stabilization would be that the average of $b$ along any high-energy bicharacterisitic of the classical Hamiltonian $p = \lambda^2$ should be bounded from below. More precisely, assuming that $b$ is continuous, these averages should be of the form
\begin{equation*}
\fint_{-t}^t b\bigl( (\pi \circ \varphi_s)(x, \xi) \bigr) \dd s
\end{equation*}
where $\pi : T^\star \R^d \to \R^d$ is the cotangent bundle projection, $(x, \xi) \in \{p = \lambda^2\}$ for $\lambda$ large, and $t > 0$ may depend on $\lambda$. In the case of a rough damping coefficient $b$, it is natural to regularize it by convolution, as we did for the Geometric Control Condition. Hence we should consider quantities of the form
\begin{equation*}
\fint_{-t}^t (b \ast \kappa_r) \left( (\pi \circ \varphi_s)(x, \xi) \right) \dd s ,
\end{equation*}
where $\kappa_r$ was introduced in~\eqref{eq:defkappa}, and the scaling factor $r > 0$ possibly depends on $\lambda$. See~\cite{BG:20,P:23,Leautaud:23} for related discussions on this question of space-regularization scale. To figure out the relevant scale for $t$ as a parameter depending on $\lambda$, we can investigate the dynamical consequence of being a $o(\lambda)$-quasimode of $P$, namely a family $(u_\lambda)_\lambda$ in $\dom P$ such that
\begin{equation} \label{eq:qmintro}
(P - \lambda^2) u_\lambda = o_{L^2}(\lambda) ,
	\qquad
\norm*{u_\lambda}_{L^2} = 1 ,
	\qquad \quad \rm{as } \;\, \lambda \to + \infty .
\end{equation}
The relevance of these quasimodes is given by Proposition~\ref{prop:quasimodes}, which states that uniform stability fails exactly when there exists such a {$o(\lambda)$-quasimode} escaping damping, namely $\inp{u_\lambda}{b u_\lambda}_{L^2} = o(1)$ as $\lambda \to + \infty$. If $(u_\lambda)_\lambda$ satisfies~\eqref{eq:qmintro}, since $\e^{\ii s P}$ is an isometry for any $s$, using the mean-value inequality, we obtain
\begin{equation} \label{eq:typicalscales}
\e^{\ii t (P - \lambda^2)} u_\lambda - u_\lambda = \int_0^t \ii \e^{\ii s (P - \lambda^2)} (P - \lambda^2) u_\lambda \dd s = o_{L^2}(t \lambda) ,
	\qquad \norm*{u_\lambda}_{L^2} = 1 .
\end{equation}
That means that $u_\lambda$ is nearly invariant under the action of the propagator $\e^{\ii t P}$ (up to a phase factor) on a time scale of order $t \approx T/\lambda$:
\begin{equation*}
\e^{\ii t P} u_\lambda
	= \e^{\ii t \lambda^2} u_\lambda + o_{L^2}(1) ,
		\qquad \textrm{uniformly for} \; t \in [0, T/\lambda] .
\end{equation*}
Here $T$ a fixed constant. As for the parameter $r$, the point is to understand what is the critical space concentration scale of $o(\lambda)$-quasimodes. This can be done roughly through the following observation, which arises when one tries to construct quasimodes~\eqref{eq:qmintro} using the WKB method (see~\cite[Chapter 2]{DS:book}). We think of this quasimode $u_\lambda$ as a Lagrangian state $u_\lambda = f_\lambda \e^{\ii \lambda \psi_\lambda}$, with envelope $f_\lambda$ (which localizes the quasimode in space) and phase $\e^{\ii \lambda \psi_\lambda}$, $\psi_\lambda : \R^d \to \R$ (which localizes the quasimode in frequency). Assume this quasimode concentrates in a tubular neighborhood of radius $r$ of a bicharacteristic projected onto the physical space (namely the $x^t$ component of a trajectory of the Hamiltonian flow; see~\eqref{eq:Hamilton}). Then the gradient of the envelope $f_\lambda$ has to be of order $1/r$, or in other words, the Fourier transform of $f_\lambda$ is mostly concentrated in a window of length $1/r$. In order $u_\lambda$ to be a $o(\lambda)$-quasimode, this frequency contribution of the envelope function should fit in the remainder $o(\lambda)$ so that it does not disrupt the microlocalization of the wave packet, and especially the frequency localization ensured by the phase $\e^{i \lambda \psi_\lambda}$. In particular, the amount of kinetic energy carried by the envelope $f_\lambda$, which is contained in the term $\e^{i \lambda \psi_\lambda} \Delta f_\lambda \approx 1/r^2$, should actually be $o(\lambda)$, and therefore the critical concentration scale appears to be $r \approx R/\sqrt{\lambda}$, with $R \gg 1$.

\subsection{Strategy of proof and plan of the paper}

Let us outline the content of the paper.

In Section~\ref{sec:background}, we introduce classical definitions from semiclassical analysis and we recall that the uniform exponential decay of the energy is equivalent to a particular resolvent estimate for the damped wave operator.

In Section~\ref{sec:aprioriNCSC}, we prove Propositions~\ref{prop:aprioriNC} and~\ref{prop:aprioriSC}, i.e.\ the a priori necessary and sufficient conditions that do not take into account the potential. The scheme of the proof is classical: for the necessary condition, we construct quasimodes (in fact wave packets) concentrating around any segment in $\R^d$. For the sufficient condition, we prove that the energy of {$o(h)$-quasimodes} of the semiclassical Laplacian escapes at infinity in space by a defect measure argument. We then deduce that the same holds for the quasimodes of the operator $P$ relevant to this problem.

Section~\ref{sec:necessarycondition} discusses the finer necessary condition of Theorem~\ref{thm:necessarycondition}. The proof is based on the construction of suitable quasimodes with an optimal space concentration near the turning points of the classical dynamics. Then we prove Corollary~\ref{cor:characterizationregulardamping} (concerning uniformly continuous damping coefficients) as a direct application. The main ingredient of this section is a relatively simple instance of the WKB method that we adapt to our setting. We find approximate solutions (quasimodes) to $(V - \frac{1}{2} \Lap  - \lambda^2) u_\lambda = 0$ under the form $u_\lambda = f_\lambda \e^{\ii \lambda \psi_\lambda}$, as we mentioned at the end of Subsection~\ref{subsec:scales} above. Since we are mainly concerned with the ``potential regime" in Theorem~\ref{thm:necessarycondition}, namely the phase space regions where the potential prevails, those quasimodes are particularly simple, since the phase function $\psi_\lambda$ is chosen to be constant equal to $0$.
Notice that our setting is quite different from the usual WKB method: we are not focusing on a fixed energy layer of a semiclassical Schrödinger operator $- h^2 \Lap + V$ as it is usually the case, but rather consider the high energy regime for $- \Lap + V$.

Finally, we deal with Proposition~\ref{prop:SCgeom} in Section~\ref{sec:stabilizationcondition}, by roughly showing that the Hamiltonian flow is well approximated by its linearization on the time scale under consideration.

We recall in Appendix~\ref{app:pseudo} classical results on the Weyl quantization and pseudo-differential calculus. Technical remarks and lemmata on conditions~\eqref{eq:GCC} and~\eqref{eq:SCdyn} are collected in Appendix~\ref{app:constant}.

\addtocontents{toc}{\SkipTocEntry}
\subsection*{Acknowledgments} I am grateful to Matthieu Léautaud for our regular discussions and his advice on this project, and for his comments on a preliminary version of this article. This research project was partly conducted while visiting Universidad Politécnica de Madrid during the academic year 2020-2021. I thank this institution for its hospitality. I also thank Fabricio Macià for numerous discussions on this topic.

\section{Toolbox of semiclassical analysis and reduction of the problem to a resolvent estimate} \label{sec:background}

Let us recall basic definitions and well-known results that will be needed afterwards.
We recall that the potential is assumed to be non-negative and confining~\eqref{eq:assumV} throughout this article.

\subsection{Weyl quantization and semiclassical defect measures}

To any function $a \in \sch(T^\star \R^d)$, we associate an operator on $\sch(\R^d)$, denoted by $\Opw{a}$ and defined by
\begin{equation} \label{eq:defquantization}
\left[\Opw{a} u\right](x)
= (2 \pi)^{-d} \int_{\R^d \times \R^d} \e^{\ii \xi \cdot (x - y)} a \left( \dfrac{x + y}{2}, \xi \right) u(y) \dd y \dd \xi ,
\qquad \forall u \in \sch(\R^d) .
\end{equation}
This operator is called a pseudo-differential operator and it maps $\sch(\R^d)$ to itself (see for instance~\cite[Theorem 4.16]{Zworski:book}). The function $a$ is called the symbol of $\Opw{a}$. See Appendix~\ref{app:pseudo} for an account of the classical properties of pseudo-differential operators.

We also introduce the semiclassical rescaling of the Weyl quantization:
\begin{equation} \label{eq:semiclassicalrescalingLap}
\Opw[h]{a}
	= \Opw{a(x, h \xi)} ,
		\qquad h \in (0, 1] .
\end{equation}

When studying particular sequences of $L^2$ functions in the semiclassical limit, semiclassical defect measures are a good tool to record their microlocalization, that is the asymptotic localization in phase space when $h \to 0^+$. For any bounded sequence $(u_h)_{h > 0}$ in $L^2(\R^d)$, the Calder\'{o}n--Vaillancourt Theorem (Theorem~\ref{thm:CV}), together with a diagonal extraction and the G{\aa}rding inequality, allows to prove that there exists a subsequence $(h_n)_{n \in \N}$ and a non-negative Radon measure $\mu$ on $T^\star \R^d$ with finite mass, which is called a \emph{semiclassical defect measure}, such that $h_n \to 0^+$ and
\begin{equation*}
\forall a \in C_c^\infty(T^\star \R^d) , \qquad \inp*{u_{h_n}}{\Opw[h_n]{a} u_{h_n}}_{L^2} \strongto{n \to \infty} \int_{T^\star \R^d} a (x, \xi) \dd \mu(x, \xi)
\end{equation*}
(we refer to~\cite[Theorem 5.2]{Zworski:book} for more details about this construction; see also~\cite{Gerard:91,LionsPaul}). When $(u_h)_{h > 0}$ has additional properties, more can be said about the limiting measure $\mu$, using pseudo-differential calculus (Theorem~\ref{thm:pseudodiffcalc}).

\begin{proposition}[{Property of semiclassical defect measures \--- \cite[Theorems 5.3-5.4]{Zworski:book}}] \label{prop:suppflowinvSDM}
Assume $(u_h)_{h > 0}$ is a bounded sequence in $L^2(\R^d)$ and let $\mu$ be a semiclassical defect measure. Then
\begin{itemize}[label=\textbullet]
\item if $(- h^2 \frac{1}{2} \Delta - 1) u_h = o_{L^2}(1)$, then $\supp \mu \subset \{ \abs{\xi} = 1 \}$;
\item if $(- h^2 \frac{1}{2} \Delta - 1) u_h = o_{L^2}(h)$, then $\mu$ is invariant by the flow $\phi^t(x, \xi) = (x + t \xi, \xi)$, namely $(\phi^t)_\ast \mu = \mu, \forall t \in \R$.
\end{itemize}
\end{proposition}

\subsection{Weyl--Heisenberg translation operators}

Given a point $\rho_0 = (x_0, \xi_0) \in T^\star \R^d$, we define $T_{\rho_0}$ to be the operator acting as
\begin{equation} \label{eq:deftranslation}
T_{\rho_0} u(x) = \e^{- \frac{\ii}{2} \xi_0 \cdot x_0} \e^{\ii \xi_0 \cdot x} u (x - x_0) , \qquad \forall u \in \sch(\R^d) .
\end{equation}
This is clearly a continuous operator mapping $\sch(\R^d)$ to itself, and it can be extended to a continuous operator $\sch'(\R^d) \to \sch'(\R^d)$. Moreover, it is unitary on $L^2(\R^d)$ with inverse $T_{-\rho_0}$. Note that we have the following identity:
\begin{equation*}
T_{\rho_0} T_{\rho_1} = \e^{\frac{\ii}{2} \sympf(\rho_0, \rho_1)} T_{\rho_0 + \rho_1} , \qquad \forall \rho_0, \rho_1 \in T^\star \R^d .
\end{equation*}
These operators allow to perform translations at the level of symbols via the equality
\begin{equation*}
T_{\rho_0} \Opw{a} T_{- \rho_0} = \Opw{a(\rho - \rho_0)} , \qquad \forall a \in \sch'(T^\star \R^d), \forall \rho_0 \in T^\star \R^d ,
\end{equation*}
which is a particular case of Egorov's Theorem (see~\cite{Zworski:book} for instance). In addition, we have the noteworthy intertwining relation with the Fourier transform:
\begin{equation} \label{eq:intertwiningr}
T_{J \rho_0} \ft = \ft T_{\rho_0} ,
	\qquad J = \begin{pmatrix} 0 & \id \\ -\id & 0 \end{pmatrix}
\end{equation}
($J$ is the usual symplectic matrix here; in particular $\sympf(\rho_0, \rho_1) = J \rho_0 \cdot \rho_1$).

\subsection{Resolvent estimate} \label{subsec:resolventestimate}

A crucial first step in our study is to reformulate the question of uniform stability as a stationary problem, under the form of a resolvent estimate. This is a common method to study the damped wave equation, see e.g.~\cite{Leb:96,AL:14,BJ:16}. In the context of uniform stability (Definition~\ref{def:unifstab}), a useful link between the evolution problem and the resolvent estimate is given by the following fundamental result, whose proof can be found in~\cite[Chapter V, Theorem 1.11]{EN:book}.

\begin{theorem}[Gearhart, Prüss, Huang, Greiner \--- {\cite{Gearhart:78,Huang:85,Pruss:84}}] \label{thm:GPHG}
A strongly continuous semigroup $\left(T(t)\right)_{t \in \R_+}$ on a Hilbert space $H$ is uniformly exponentially stable if and only if the half plane $\set{z \in \CC}{\Re z \ge 0}$ is contained in the resolvent set of the generator $A$, with the resolvent satisfying
\begin{equation*}
\sup_{\lambda \in \R} \norm*{(A - \ii \lambda)^{-1}}_{\Bop(H)} < \infty .
\end{equation*}
\end{theorem}

Applying this theorem to the damped wave operator $\dwop_b$, we deduce a characterization of the uniform stability of~\eqref{eq:eq}. In fact, it is more convenient to work with the converse statement, characterizing the failure of uniform stability.

\begin{proposition} \label{prop:quasimodes}
Assume $V$ is subject to~\eqref{eq:assumV} and that the non-negative damping coefficient $b \in L^\infty(\R^d)$ does not identically vanish. The equation~\eqref{eq:eq} is not uniformly stable if and only if there exists a sequence $(u_n)_{n \in \N}$ in $\dom P$ with $\norm*{u_n}_{L^2} = 1, \forall n \in \N$, and a sequence of real numbers $\lambda_n \to + \infty$, such that
\begin{enumerate}
\item\label{it:quasimode} $(u_n)_{n \in \N}$ is a \emph{$o(\lambda_n)$-quasimode} of $P$, namely $(P - \lambda_n^2) u_n = o_{L^2}(\lambda_n)$ as $n \to \infty$;
\item\label{it:asymptundamped} $\inp*{u_n}{b u_n}_{L^2} \to 0$ as $n \to \infty$ (the $u_n$'s are ``asymptotically undamped").
\end{enumerate}
\end{proposition}

\begin{proof}
%
%
By~\cite[Lemma 4.6]{AL:14}, the resolvent of $\dwop_b$ can be controlled by the resolvent of the operator
\begin{equation*}
P_b(z) = P + z^2 + z b \, \qquad \dom P_b(z) = \dom P ,
\end{equation*}
near infinity on the imaginary axis: there exist $C\ge 1$ and $\lambda_0 > 0$ such that
\begin{equation*}
\forall \lambda \in \R, \abs*{\lambda} \ge \lambda_0 , \qquad
C^{-1} \abs*{\lambda} \norm*{P_b(\ii \lambda)^{-1}}_{\Bop(L^2)} \le \norm*{(\dwop_b - \ii \lambda)^{-1}}_{\Bop(\hilbert)} \le C \left( 1 + \abs*{\lambda} \right) \norm*{P_b(\ii \lambda)^{-1}}_{\Bop(L^2)} .
\end{equation*}
In addition, since the norm of the resolvent is a continuous function, we know that $\norm*{(\dwop_b - \ii \lambda)^{-1}}_{\Bop(\hilbert)}$ is bounded for any $\lambda$ in a fixed compact set. Therefore Theorem~\ref{thm:GPHG} implies that the equation~\eqref{eq:eq} is not uniformly stable if and only if there exist a sequence $(\lambda_n)_{n \in \N}$ tending to $\pm \infty$ such that $\abs*{\lambda_n} \norm*{P_b(\ii \lambda_n)^{-1}}_{\Bop(\hilbert)} \to + \infty$ as $n \to \infty$. Without loss of generality, one can assume that $\lambda_n \to + \infty$ since $P_b(\ii \lambda)^\ast = P_b(- \ii \lambda)$ for all $\lambda \in \R$. We infer that there exists a sequence $(u_n)_{n \in \N}$ in $\dom P$ such that $\norm*{u_n}_{L^2} = 1, \forall n \in \N$ and $P_b(\ii \lambda_n) u_n = o(\lambda_n)$. Then we compute:
\begin{equation*}
o(\lambda_n)
= \Im \inp*{u_n}{P_b(\ii \lambda_n) u_n}_{L^2}
= \Im \left( \inp*{u_n}{(P - \lambda_n^2) u_n}_{L^2} + \ii \lambda_n \inp*{u_n}{b u_n}_{L^2} \right)
= \lambda_n \inp*{u_n}{b u_n}_{L^2} ,
\end{equation*}
which yields $\inp*{u_n}{b u_n}_{L^2} = \norm*{\sqrt{b} u_n}_{L^2}^2 \to 0$ in $L^2(\R^d)$. Since $b$ is $L^\infty$, we deduce that $b u_n \to 0$ in $L^2$, and from
\begin{equation*}
(P - \lambda_n^2) u_n = P_b(\ii \lambda_n) u_n - \ii \lambda_n b u_n ,
\end{equation*}
we deduce that
\begin{equation*}
\left((P - \lambda_n^2) u_n
	= o_{L^2}(\lambda_n)
		\quad \rm{and} \quad
\sqrt{b} u_n
	= o_{L^2}(1) \right)
		\qquad \Longleftrightarrow \qquad
P_b(\ii \lambda_n) u_n
	= o_{L^2}(\lambda_n)
\end{equation*}
as $n \to \infty$, which completes the proof.
\end{proof}

In view of this characterization, our goal is to study the properties of localization and oscillation of $o(\lambda_n)$-quasimodes and see under which conditions on the damping coefficient they are ``asymptotically undamped" in the sense of Item~\ref{it:asymptundamped} of Proposition~\ref{prop:quasimodes}. For simplicity, we will drop the subscript $n$ and call $o(\lambda)$-quasimode of $P$ any family of functions $(u_\lambda)_\lambda$ lying in $\dom P$ and normalized in $L^2(\R^d)$, such that
\begin{equation} \label{eq:quasimode}
(P - \lambda^2) u_\lambda = o_{L^2}(\lambda) , \qquad \lambda \to + \infty .
\end{equation}

\paragraph*{\textbf{Comments on the natural semiclassical setting for homogeneous potentials.}} Note that $o(\lambda)$-quasimodes can be recast in a semiclassical framework if the potential $V$ is homogeneous, by choosing an appropriate scaling of the phase space (see~\cite[Section 1.7]{P:23} for a similar discussion). If $V(x) = \abs*{x}^{2m}$ where $m \ge 0$ is a real number, one may introduce the unitary dilation operator $\Lambda_h$ mapping $u \in L^2(\R^d)$ to the function $x \mapsto h^{d \gamma/2} u\left(h^\gamma x\right)$ with $\gamma = 1/(m + 1)$, and then define a quantization $\widetilde{\quantization}_h$ by
\begin{equation} \label{eq:semiclassicalrescalingpot}
\widetilde{\quantization}_h(a)
	:= \Opw[1]{a\left(h^\gamma x, h^{1 - \gamma} \xi\right)}
	= \Lambda_h \Ophw{a} \Lambda_h^\ast , \qquad \forall a \in \sch(T^\star \R^d) ,
\end{equation}
where $\Ophw{a}$ denotes the standard semiclassical Weyl quantization of the symbol $a$ recalled in~\eqref{eq:semiclassicalrescalingLap}. In this framework, substituting $1/\lambda$ for $h^{1 - \gamma}$ and relabeling $u_{1/h^{1-\gamma}}$ as $u_h$, \eqref{eq:quasimode} becomes:
\begin{equation} \label{eq:sclPm}
\left(h^\frac{2m}{m+1} P - 1\right) u_h = \widetilde{\quantization}_h(p - 1) u_h = o_{L^2}(h^{1 - \gamma}) ,
	\qquad h \to 0^+ ,
\end{equation}
namely $(u_h)_{h > 0}$ is a semiclassical quasimode with precision $o(h^{1 - \gamma})$. Notice that this formulation is commonly used to study the high-frequency asymptotics for the operator $P = - \frac{1}{2} \Lap + 1$; see~\cite{AL:14,BJ:16,Zworski:book}. Since in that case the potential $V$ is bounded, it fits into the $o(\lambda)$ remainder, so that setting $h = 1/\lambda$, we obtain
\begin{equation} \label{eq:quasimodeburqjoly}
\left(- h^2 \tfrac{1}{2} \Lap - 1\right) u_h = o_{L^2}(h) .
\end{equation}
Thus we observe that the case $P =  - \frac{1}{2} \Lap + 1$ yields ``good quasimodes", since a precision $o(h)$ allows to provides a fairly accurate description of the asymptotic behavior of $u_h$ as $h \to 0$ (Proposition~\ref{prop:suppflowinvSDM}). Quasimodes of the form~\eqref{eq:quasimodeburqjoly} are related to~\eqref{eq:GCC}, since the Hamiltonian flow associated with the Laplacian is the geodesic flow (here straight lines). In comparison, in the presence of a confining potential, quasimodes are coarser since their precision is $o(h^{1 - \gamma})$ with $\gamma \in (0, 1)$. This phenomenon highlights one of the difficulties of the study of stabilization in the presence of a confining potential: the asymptotic invariance properties of $u_\lambda$ as $\lambda \to \infty$ will be slightly more difficult to exhibit than in the ``free case" $V(x) = 1$ where the dynamics is driven by the Laplacian only. Sharp concentration properties of quasimodes appear at \emph{second-microlocal} scales.


\section{A priori conditions of stabilization: Proof of Propositions~\ref{prop:aprioriNC} and~\ref{prop:aprioriSC}} \label{sec:aprioriNCSC}

This section is devoted to the proof of Propositions~\ref{prop:aprioriNC} and~\ref{prop:aprioriSC}. In Subsection~\ref{subsec:beamtype}, we give a quasimode construction which allows to prove the necessary condition of Prooposition~\ref{prop:aprioriNC}. Then in Subsection~\ref{subsec:escapeofmass}, we show that the energy of quasimodes escapes from any compact set and we prove the sufficient condition of Proposition~\ref{prop:aprioriSC}.

\subsection{A priori necessary condition} \label{subsec:beamtype}

In order to investigate the necessary condition of uniform stability, we construct {$o(\lambda)$-quasimodes} for $P$ by constructing quasi-modes for the Laplacian, whose potential energy fits in the remainder term. This is done by using wave-packets concentrating along straight lines, giving rise to the necessity of~\eqref{eq:GCC}.

\begin{proposition}[Construction of kinetic quasimodes] \label{prop:constructionbeam}
Let $P$ be the operator defined in~\eqref{eq:defP} with a potential $V$ subject to~\eqref{eq:assumV}. Let $(x_n, \nu_n)_{n \in \N}$ be a sequence in $S \R^d$, and let $(t_n)_{n \in \N}$ and $(r_n)_{n \in \N}$ be sequences of real positive numbers such that $t_n \to + \infty$ and $r_n \to 0$ as $n \to \infty$. Define a sequence $(\lambda_n)_{n \in \N}$ by
\begin{equation*}
\lambda_n = \max\left\{ \dfrac{(n + 1)^2}{r_n^2} , n \norm*{V}_{L^\infty\left(B_{t_n}(x_n)\right)} \right\} , \qquad \forall n \in \N .
\end{equation*}
Note that $\lambda_n \to + \infty$ as $n \to \infty$. Set $\rho_n = (x_n, \lambda_n \nu_n)$, and denote by $\Sigma_n$ the automorphism of $\R^d$ defined by
\begin{equation*}
\Sigma_n = t_n \ket{\nu_n} \bra{\nu_n} + \dfrac{n + 1}{\sqrt{\lambda_n}} \left( \id - \ket{\nu_n} \bra{\nu_n} \right) ,
\end{equation*}
where $\ket{\nu_n} \bra{\nu_n}$ (resp. $\id - \ket{\nu_n} \bra{\nu_n}$) is the orthogonal projection onto $\Span \nu_n$ (resp. $\Span \nu_n{}^\perp$). Further denote by $M_n$ the unitary dilation operator acting on $L^2(\R^d)$ as $(M_n f)(x) = \abs*{\det \Sigma_n}^{-1/2} f(\Sigma_n^{-1} x)$, and recall that $T_{\rho_n}$ is the phase-space translation operator~\eqref{eq:deftranslation} associated with $\rho_n$. Then if we fix $k \in \cont_\comp^\infty(\R^d)$ satisfying $\norm*{k}_{L^2} = 1$ and $\supp k \subset B_1(0)$, the family of functions
\begin{equation*}
u_n := T_{\rho_n} M_n k , \qquad n \in \N ,
\end{equation*}
forms a $o(\lambda_n)$-quasimode of $P$, that is $\norm*{u_n}_{L^2} = 1, \forall n \in \N$ and $(P - \lambda_n^2) u_n = o_{L^2}(\lambda_n)$ as $n \to \infty$.
\end{proposition}

\begin{proof}
First note that since $T_{\rho_n}$ and $M_n$ are unitary, the $u_n$'s are normalized in $L^2$, and they are also $\cont_\comp^\infty$, so that they belong to the domain of $P$. In the proof, we will use the following observation:
\begin{equation} \label{eq:normSigman}
\norm*{\Sigma_n}_{\Bop(\R^d)} \le t_n \qquad \rm{and} \qquad \norm*{\Sigma_n^{-1}}_{\Bop(\R^d)} \le \dfrac{\sqrt{\lambda_n}}{n + 1} = o(\sqrt{\lambda_n}) \qquad \textrm{for $n$ large enough} .
\end{equation}
It essentially comes from the assumption $t_n \to + \infty$ and the fact that $(n + 1)/\sqrt{\lambda_n} \le r_n \to 0$ as $n \to \infty$.

Then notice that the potential term $V u_n$ is negligible. Indeed, the operators $T_{\rho_n}$ and $M_n$ being unitary, we have
\begin{align*}
\norm*{V u_n}_{L^2}
	&= \norm*{M_n^\ast T_{\rho_n}^\ast V u_n}_{L^2}
	= \norm*{M_n^\ast T_{\rho_n}^\ast V T_{\rho_n} M_n k}_{L^2}
	= \norm*{V(x_n + \Sigma_n x) k(x)}_{L_x^2} \\
	&\le \norm*{V}_{L^\infty(x_n + \Sigma_n \supp k)} \norm*{k}_{L^2} .
\end{align*}
Yet $\supp k \subset B_1(0)$, so that using~\eqref{eq:normSigman}, $x_n + \Sigma_n \supp k \subset B_{t_n}(x_n)$. It follows that
\begin{equation} \label{eq:Vun}
\norm*{V u_n}_{L^2}
	\le \norm*{V}_{L^\infty\left(B_{t_n}(x_n)\right)}
	\le \lambda_n/n
	= o(\lambda_n)
\end{equation}
by definition of $\lambda_n$. Thus in order to conclude the proof of the proposition, we would like to prove that $(u_n)_{n \in \N}$ is a quasimode of the Laplacian. To do this, it will be more convenient to proceed on the Fourier side, and prove that $(\abs*{\xi}^2 - \lambda_n^2) \ft u_n = o_{L^2}(\lambda_n)$. Recall from~\eqref{eq:intertwiningr} that the translation operators enjoy the intertwining relations $\ft T_{\rho_n} = T_{J \rho_n} \ft$, and also that for dilation operators $\ft M_n = M_n^\ast \ft$, so that $\ft u_n = T_{J \rho_n} M_n^\ast \ft k$. Accordingly,
\begin{align*}
M_n T_{J \rho_n}^\ast \left( \abs*{\xi}^2 - \lambda_n^2 \right) \ft u_n(\xi)
	&= M_n T_{J \rho_n}^\ast \left( \abs*{\xi}^2 - \lambda_n^2 \right) T_{J \rho_n} M_n^\ast \ft k(\xi)
	= \left( \abs*{\lambda_n \nu_n + \Sigma_n^{-1} \xi}^2 - \lambda_n^2 \right) \ft k(\xi) \\
	&= \left( 2 \lambda_n \Sigma_n^{-1} \nu_n \cdot \xi  + \abs*{\Sigma_n^{-1} \xi}^2 \right) \ft k(\xi) ,
\end{align*}
where we used the fact that $\Sigma_n^{-1}$ is symmetric for the last equality. Finally, by definition of $\Sigma_n$, it holds $\Sigma_n^{-1} \nu_n = \nu_n/t_n = o(1)$, and as we have seen in~\eqref{eq:normSigman}, we have $\norm*{\Sigma_n^{-1}}_{\Bop(\R^d)} = o(\sqrt{\lambda_n})$. Therefore we obtain
\begin{equation*}
\norm*{M_n T_{J \rho_n}^\ast \left( \abs*{\xi}^2 - \lambda_n^2 \right) \ft u_n}_{L^2}
	\le  2 \lambda_n \abs*{\Sigma_n^{-1} \nu_n} \norm*{\xi \ft k}_{L^2} + \norm*{\Sigma_n^{-1}}_{\Bop(\R^d)}^2 \norm*{\abs*{\xi}^2 \ft k}_{L^2}
	= o(\lambda_n) .
\end{equation*}
Since, again, the translation and dilation operators are unitary, we conclude that $(\abs*{\xi}^2 - \lambda_n^2) \ft u_n = o_{L^2}(\lambda_n)$. The Plancherel theorem together with~\eqref{eq:Vun} finish the proof.
\end{proof}

We are now ready to prove that~\eqref{eq:GCC} is necessary for stabilization.

\begin{proof}[Proof of Proposition~\ref{prop:aprioriNC}]
We argue by contradiction: assume~\eqref{eq:GCC} is not true. Since~\eqref{eq:GCC} is equivalent to~\eqref{eq:GCCnear0} in view of Remark~\ref{rmk:GCCnear0}, we deduce that for any $n \in \N^\star$, there exist $(x_n, \nu_n) \in S \R^d$ and 
\begin{equation} \label{eq:conditionrn}
0 < r_n \le \min\left\{ \dfrac{1}{n}, \sqrt{\dfrac{n}{\norm*{V}_{L^\infty\left(B_{n}(x_n)\right)}}} \right\}
\end{equation}
satisfying
\begin{equation} \label{eq:avglessthan2-n}
\fint_{-n}^n (b \ast \kappa_{r_n}) (x_n + t \nu_n) \dd t \le 2^{-n} .
\end{equation}
We will construct a $o(\lambda)$-quasimode of $P$ concentrating around the rays $\set{x_n + t \nu_n}{t \in [-n, n]}$ which is not asymptotically undamped in the sense of Item~\ref{it:asymptundamped} of Proposition~\ref{prop:quasimodes}. In $\R^d$, we define the cylinder with axis $\R \nu_n$, center $x_n$, length $n$ and radius $\sqrt{3} r_n/2$ by:
\begin{equation*}
\scr{C}_n = \set{y \in \R^d}{\dist(y - x_n, \R \nu_n) \le \dfrac{\sqrt{3}}{2} r_n, \abs*{(y - x_n) \cdot \nu_n} \le \dfrac{n}{2}} ,
\end{equation*}
as well as the (time) interval
\begin{equation*}
I_n(y) = \left[ (y - x_n) \cdot \nu_n - \dfrac{r_n}{2}, (y - x_n) \cdot \nu_n + \dfrac{r_n}{2} \right] \subset \R ,
\end{equation*}
which depends on $y \in \R^d$. When $y \in \scr{C}_n$, we observe that $I_n(y) \subset [-n, n]$ (recall that $r_n \le 1/n \le n$ for any $n \in \N^\star$). In addition, if $y \in \scr{C}_n$ and $t \in I_n(y)$, using the Pythagorean Theorem, it holds
\begin{align*}
\abs*{x_n - y + t \nu_n}^2
	&= \abs*{\left((x_n - y) \cdot \nu_n\right) \nu_n + t \nu_n}^2 + \abs*{x_n - y - \left((x_n - y) \cdot \nu_n\right) \nu_n}^2 \\
	&= \abs*{t - (y - x_n) \cdot \nu_n}^2 + \dist(y - x_n, \R \nu_n)^2
	\le \left(\dfrac{r_n}{2}\right)^2 + \left(\dfrac{\sqrt{3}}{2} r_n\right)^2
	= r_n^2 ,
\end{align*}
which can be reformulated as
\begin{equation} \label{eq:supptent}
y \in \scr{C}_n, t \in I_n(y) \qquad \Longrightarrow \qquad \kappa_{r_n}(x_n - y + t \nu_n) = \dfrac{1}{r_n^d \abs{B_1(0)}} ,
\end{equation}
where $\kappa_r$ is introduced in~\eqref{eq:defkappa}.
Now from~\eqref{eq:avglessthan2-n} and the definition of the convolution, we have
\begin{align*}
2^{-n}
	&\ge \fint_{-n}^n (b \ast \kappa_{r_n}) (x_n + t \nu_n) \dd t
	= \int_{\R^d} b(y) \fint_{-n}^n \kappa_{r_n}(x_n - y + t \nu_n) \dd t \dd y \\
	&\ge \int_{\scr{C}_n} b(y) \dfrac{1}{2n} \int_{I_n(y)} \kappa_{r_n}(x_n - y + t \nu_n) \dd t \dd y .
\end{align*}
Then by~\eqref{eq:supptent} we obtain
\begin{equation*}
2^{-n}	
	\ge \dfrac{1}{2 n r_n^{d-1} \abs{B_1(0)}} \int_{\scr{C}_n} b(y) \dd y
\end{equation*}
(recall that $I_n(y)$ has length $r_n$),
which yields
\begin{equation*}
\int_{\scr{C}_n} b(y) \dd y \le 2 \abs{B_1(0)} n r_n^{d-1} 2^{-n} , \qquad \forall n \in \N^\star .
\end{equation*}
Now we consider $(u_n)_{n \in \N}$ the quasimode of $P$ provided by Proposition~\ref{prop:constructionbeam} (with $t_n = n/2$ and $r_n/2$ in place of $r_n$). It comes together with a sequence $(\lambda_n)_{n \in \N}$ so that $(P - \lambda_n^2) u_n = o_{L^2}(\lambda_n)$. Thanks to~\eqref{eq:conditionrn}, the definition of $\lambda_n$ gives
\begin{equation*}
\lambda_n = \left( \dfrac{n + 1}{r_n} \right)^2 , \qquad \forall n \in \N ,
\end{equation*}
and recalling the definition of $u_n$, we see that
\begin{equation*}
\supp u_n \subset \scr{C}_n \qquad \rm{and} \qquad \norm*{u_n}_\infty = \dfrac{\norm*{k}_\infty}{\sqrt{t_n}} r_n^{-\frac{d-1}{2}} .
\end{equation*}
We conclude that
\begin{align*}
\abs*{\inp*{u_n}{b u_n}_{L^2}}
	&= \int_{\scr{C}_n} b(y) \abs*{u_n(y)}^2 \dd y
	\le \norm*{u_n}_\infty^2 \int_{\scr{C}_n} b(y) \dd y
	\le \dfrac{\norm*{k}_\infty^2}{t_n} \times 2 \abs{B_1(0)} n 2^{-n} \\
	&= 4 \norm*{k}_\infty^2 \abs{B_1(0)} 2^{-n}
	\strongto{n \to \infty} 0 .
\end{align*}
Therefore by Proposition~\ref{prop:quasimodes}, \eqref{eq:eq} is not uniformly stable, which is the sought result.
\end{proof}

\subsection{A priori sufficient condition} \label{subsec:escapeofmass}

We now prepare for the proof of the a priori sufficient condition, which involves semiclassical defect measures associated with quasimodes of the Laplacian. We recall that $\phi^t(x, \xi) = (x + t \xi, \xi)$ is the Hamiltonian flow associated with the symbol $(x, \xi) \mapsto \frac{1}{2} \abs{\xi}^2$.

\begin{lemma} \label{lem:mu=0Lap}
Let $\mu$ be a nonnegative Radon measure on $T^\star \R^d$ having finite mass and satisfying
\begin{equation*}
\supp \mu \subset \left\{ \frac{1}{2} \abs*{\xi}^2 = 1 \right\} \qquad \rm{and} \qquad (\phi^t)_\ast \mu = \mu , \; \forall t \in \R .
\end{equation*}
Then $\mu = 0$.
\end{lemma}

\begin{proof}
Fix $a \in \cont_\comp^\infty(T^\star \R^d)$, and let us show that $\int_{T^\star \R^d} a \dd \mu = 0$. Let $\chi \in \cont_\comp^\infty(\R_+^\star)$ satisfy $\chi(s) = 1$ in a neighborhood of $\sqrt{2}$, in order that $\chi(\abs*{\xi}) = 1$ on the support of $\mu$. Then we have
\begin{equation*}
\int_{T^\star \R^d} a \dd \mu
	= \int_{T^\star \R^d} a(x, \xi) \chi(\abs*{\xi}) \dd \mu(x, \xi)
	= \int_{T^\star \R^d} a(x + t \xi, \xi) \chi(\abs*{\xi}) \dd \mu(x, \xi) ,
\end{equation*}
where we used the flow invariance of $\mu$ and $\abs{\xi}$ in the last equality, for any $t \in \R$. By averaging this on $[-T, T]$, for some $T > 0$, we obtain
\begin{equation} \label{eq:intdomain}
\int_{T^\star \R^d} a \dd \mu
	= \int_{T^\star \R^d} \left( \fint_{-T}^T a(x + t \xi, \xi) \dd t \right)\chi(\abs*{\xi}) \dd \mu(x, \xi) .
\end{equation}
Since $a$ is compactly supported, there exists $R > 0$ such that $\abs*{x} \ge R$ implies $a(x, \xi) = 0$ for all $\xi \in (\R^d)^\star$. Also introduce $r := \inf \supp \chi > 0$. Then one can restrict the integration domain in~\eqref{eq:intdomain} to the set of $(x, \xi, t) \in T^\star \R^d \times \R$ such that
\begin{equation*}
\abs*{\xi} \ge r
	\qquad \rm{and} \qquad
\abs*{x + t \xi} \le R .
\end{equation*}
Thanks to the reverse triangle inequality, we have for such $(x, \xi, t)$:
\begin{equation} \label{eq:constraintt}
\abs*{\abs*{t} - \dfrac{\abs*{x}}{\abs*{\xi}}}
	= \dfrac{1}{\abs{\xi}} \Bigl|\abs*{t \xi} - \abs*{x}\Bigr|
	\le \dfrac{\abs{t \xi + x}}{\abs{\xi}}
	\le \dfrac{R}{r} .
\end{equation}
We deduce that for fixed $(x, \xi)$, the Lebesgue measure of the set of $t$'s satisfying the constraint~\eqref{eq:constraintt} is less than $2 R/r$. Therefore
\begin{equation*}
\abs*{\int_{T^\star \R^d} a \dd \mu}
	\le \dfrac{1}{2T} \int_{T^\star \R^d} 2 \dfrac{R}{r} \norm*{a}_\infty \abs*{\chi(\abs*{\xi})} \dd \mu(x, \xi)
	\le \dfrac{R}{Tr} \norm*{a}_\infty \norm*{\chi}_\infty \mu(T^\star \R^d) .
\end{equation*}
The constants $R$ and $r$ depend only on $a$ and $\chi$ respectively, so we obtain the desired conclusion by letting $T \to \infty$.
\end{proof}

The next step towards the sufficient condition is to prove that the energy of a semiclassical $o(h)$-quasimode of the Laplacian always escapes from any compact set in the limit $h \to 0$.

\begin{proposition} \label{prop:escapeLap}
Let $(u_h)_{h > 0}$ be a bounded family in $L^2(\R^d)$ satisfying $(- \frac{h^2}{2} \Lap - 1) u_h = o_{L^2}(h)$. Then
\begin{equation*}
\int_{\R^d} a \abs*{u_h}^2 \dd x \strongto{h \to 0} 0 , \qquad \forall a \in \cont_\comp^\infty(\R^d) .
\end{equation*}
\end{proposition}

\begin{proof}
Thanks to Proposition~\ref{prop:suppflowinvSDM}, we know that any semiclassical defect measure associated with $(u_h)_{h > 0}$ satisfies the hypotheses of Lemma~\ref{lem:mu=0Lap}, and therefore is the zero measure. This means that
\begin{equation} \label{eq:wignerdistrtendtozero}
\inp*{u_h}{\Ophw{a} u_h}_{L^2} \strongto{h \to 0} 0 , \qquad \forall a \in \cont_\comp^\infty(T^\star \R^d)
\end{equation}
(notice that this limit holds for the whole family $(u_h)_{h > 0}$, and not only along a subsequence, by uniqueness of the semiclassical defect measure). Now fix $a \in \cont_\comp^\infty(\R^d)$. To finish the proof, we will use a cut-off function in the $\xi$ variable in order to reduce to a compactly supported symbol on $T^\star \R^d$ and be able to apply~\eqref{eq:wignerdistrtendtozero}. Let $\chi \in \cont_\comp^\infty(\R_+^\star)$ satisfy $\chi(s) = 1$ in a neighborhood of $\sqrt{2}$, so that $\chi(\abs{\xi})$ has compact support in $\xi$ and equals~$1$ in a neighborhood of $\{ \frac{1}{2} \abs*{\xi}^2 = 1 \}$. Then we have
\begin{equation*}
\Ophw{(1 - \chi) a} u_h
	= \Ophw{\dfrac{(1 - \chi) a}{\frac{1}{2} \abs*{\xi}^2 - 1}} \Ophw{\frac{1}{2} \abs*{\xi}^2 - 1} u_h + O_{L^2}(h)
	= o_{L^2}(h) + O_{L^2}(h)
	= o_{L^2}(1) .
\end{equation*}
For the first equality, we have used the pseudo-differential calculus (Theorem~\ref{thm:pseudodiffcalc}) with $a_1 = (1 - \chi) a/(\frac{1}{2} \abs*{\xi}^2 - 1)$ which is a well-defined symbol in $S(\jap{\xi}^{-2})$ since $1 - \chi$ vanishes in a neighborhood of $\{\frac{1}{2} \abs*{\xi}^2 = 1\}$, and $a_2 = \frac{1}{2} \abs*{\xi}^2 - 1 \in S(\jap{\xi}^2)$. Then we have used the assumption on $(u_h)_{h > 0}$ together with the Calder\'{o}n--Vaillancourt Theorem (Theorem~\ref{thm:CV}) to get the second equality. Finally, combining this with~\eqref{eq:wignerdistrtendtozero}, we obtain
\begin{equation*}
\int_{\R^d} a \abs*{u_h}^2 \dd x
	= \inp*{u_h}{\Ophw{\chi a} u_h}_{L^2} + \inp*{u_h}{\Ophw{(1 - \chi) a} u_h}_{L^2}
	\strongto{h \to 0} 0 ,
\end{equation*}
which is the sought result.
\end{proof}

In the following proposition, we show that the escape at infinity of the energy still occurs for any $o(\lambda)$-quasimode of $P$, which will pave the way for the subsequent a priori sufficient condition. The basic idea is to turn a given quasimode of $P$ into a semiclassical quasimode of the Laplacian by means of a suitable space cut-off.

\begin{proposition}[Escape from any compact set] \label{prop:escape}
Let $P$ be the operator defined in~\eqref{eq:defP} with a potential $V$ subject to~\eqref{eq:assumV}. Let $(u_\lambda)_\lambda$ be a $o(\lambda)$-quasimode of $P$. Then it holds:
\begin{equation*}
\forall a \in \cont_\comp^\infty(\R^d), \qquad \int_{\R^d} a(x) \abs*{u_\lambda(x)}^2 \dd x \strongto{\lambda \to + \infty} 0 .
\end{equation*}
\end{proposition}

\begin{proof}
Let $(u_\lambda)_\lambda$ be a {$o(\lambda)$-quasimode} of $P$. Fix $a \in \cont_\comp^\infty(\R^d)$. We set $h = 1/\lambda$ and $u_h = u_\lambda$ so that\footnote{Notice that this corresponds to the semiclassical scaling associated with the free Laplacian~\eqref{eq:semiclassicalrescalingLap} and not the one involving a homogeneous potential~\eqref{eq:semiclassicalrescalingpot}.}
\begin{equation*}
( h^2 P - 1) u_h = o_{L^2}(h) \qquad \rm{as} \; h \to 0 .
\end{equation*}
Pick a cut-off function $\psi \in \cont_\comp^\infty(\R^d)$ which is identically equal to $1$ in a neighborhood of the origin, in order that $\psi_R(x) = \psi(x/R)$ is identically $1$ on $\supp a$ for all $R$ large enough. Define for any $h > 0$:
\begin{equation*}
R_h = \sup \set{R > 0}{\norm*{\psi_R V}_{L^\infty} \le \dfrac{1}{\sqrt{h}}} ,
\end{equation*}
which is a finite positive number for all $h$ sufficiently small. It follows from this definition that $\norm*{\psi_{R_h} V}_{L^\infty} \le 1/\sqrt{h}$ and that $R_h \to \infty$ as $h \to 0$ (recall that $V$ is a confining potential under~\eqref{eq:assumV}). Therefore, setting $v_h = \psi_{R_h} u_h$, we have
\begin{align} \label{eq:computationWKB}
\left(- \dfrac{h^2}{2} \Lap - 1\right) v_h
	&= \left(- \dfrac{h^2}{2} \Lap \psi_{R_h}\right) u_h - h^2 \nabla \psi_{R_h} \cdot \nabla u_h + \psi_{R_h} (h^2 P - 1) u_h - h^2 \psi_{R_h} V u_h \nonumber\\
	&= O_{L^2}\left(\dfrac{h^2}{R_h^2}\right) - \dfrac{h}{R_h} (\nabla \psi) \left(\dfrac{x}{R_h}\right) \cdot h \nabla u_h + o_{L^2}(h) + O_{L^2}\left(h^2 \norm*{\psi_{R_h} V}_{L^\infty}\right) \nonumber\\
	&= - \dfrac{h}{R_h} (\nabla \psi) \left(\dfrac{x}{R_h}\right) \cdot h \nabla u_h + o_{L^2}(h) .
\end{align}
Yet recall that
\begin{equation*}
\tfrac{1}{2} \norm*{h \nabla u_h}_{L^2}^2
	\le h^2 \left(\tfrac{1}{2} \norm*{\nabla u_h}_{L^2}^2 + \norm*{V^{1/2}u_h}_{L^2}^2\right)
	= \norm*{u_h}_{L^2}^2 + \inp*{(h^2 P - 1) u_h}{u_h}_{L^2}
	= 1 + o_{L^2}(h)
\end{equation*}
($u_h \in \dom P$). Together with~\eqref{eq:computationWKB}, this implies that $(v_h)_{h > 0}$ (which is a bounded family of $L^2$ functions) is a semiclassical $o(h)$-quasimode of the Laplacian, that is to say
\begin{equation*}
\left(- \dfrac{h^2}{2} \Lap - 1\right) v_h = o_{L^2}(h) \qquad \rm{as} \; h \to 0 .
\end{equation*}
Now Proposition~\ref{prop:escapeLap} yields
\begin{equation} \label{eq:defectmeasurezero}
\int_{\R^d} a \abs*{v_h}^2 \dd x \strongto{h \to 0} 0 , \qquad \forall a \in \cont_\comp^\infty(\R^d) .
\end{equation}
The result follows from the fact that $v_h(x) = u_h(x)$ on $\supp a$ for $h$ small enough.
\end{proof}

Now we prove Proposition~\ref{prop:aprioriSC}.

\begin{proof}[Proof of Proposition~\ref{prop:aprioriSC}]
We show the converse statement: suppose the damped wave equation~\eqref{eq:eq} is not uniformly stable. Then by Proposition~\ref{prop:quasimodes} there exists a {$o(\lambda)$-quasimode} of $P$ (see Item~\ref{it:quasimode}) denoted by $(u_\lambda)_\lambda$, which is ``asymptotically undamped" (see Item~\ref{it:asymptundamped}), that is to say $\inp*{u_\lambda}{b u_\lambda}_{L^2} \to 0$ as $\lambda \to \infty$. Let $R > 0$ and $a \in \cont_\comp^\infty(\R^d)$ be such that $0 \le a(x) \le 1$ for any $x \in \R^d$ and $a(x) = 1$ on $B_R(0)$. Then applying Proposition~\ref{prop:escape} we see that
\begin{equation*}
\abs*{\int_{B_R(0)} b \abs*{u_\lambda}^2 \dd x}
	\le \norm*{b}_{L^\infty} \int_{B_R(0)} \abs*{u_\lambda}^2 \dd x
	\le \norm*{b}_{L^\infty} \int_{\R^d} a \abs*{u_\lambda}^2 \dd x
	\strongto{\lambda \to \infty} 0 ,
\end{equation*}
which implies that
\begin{equation*}
o(1)
	= \inp*{u_\lambda}{b u_\lambda}_{L^2}
	= o(1) + \int_{\R^d \setminus B_R(0)} b \abs*{u_\lambda}^2 \dd x .
\end{equation*}
Writing $b_R = \essinf_{\R^d \setminus B_R(0)} b$, we have, applying Proposition~\ref{prop:escape} once more
\begin{equation*}
o(1)
	= \int_{\R^d \setminus B_R(0)} b \abs*{u_\lambda}^2 \dd x
	\ge b_R \int_{\R^d \setminus B_R(0)} \abs*{u_\lambda}^2 \dd x
	= b_R \left(1 + o(1)\right)
\end{equation*}
as $\lambda \to \infty$, so that $b_R = 0$. Notice that $R > 0$ is arbitrary. Therefore~\eqref{eq:boundedoutsidecompactset} is not true, which concludes the proof.
\end{proof}

\section{Proof of Theorem~\ref{thm:necessarycondition}} \label{sec:necessarycondition}

In this section we prove Theorem~\ref{thm:necessarycondition}. The rough idea is to construct quasimodes which concentrate in space as much as possible in order to derive the stronger possible necessary condition on $b$.
%
%
We start with a technical lemma quantifying the growth of the potential.

\begin{lemma} \label{lem:lemV}
Let $V$ satisfy~\eqref{eq:assumV} and Assumption~\eqref{assum:assumptionspotentialdampedeq}. Then there exists a non-increasing function $\epsilon : \R_+^\star \mapsto \R_+^\star$ tending to zero at infinity, and depending only on $V$, such that the following assertion holds: for any $\lambda \ge 1$ and $x_0 \in \R^d$ satisfying $V(x_0) \le \lambda^2$ we have
\begin{equation*}
\forall x \in \bar B_{\sqrt{\lambda}}(x_0), \qquad \abs*{V(x) - V(x_0)} \le \abs*{x - x_0} \lambda^{3/2} \epsilon(\lambda) .
\end{equation*}
\end{lemma}

\begin{proof}
Set
\begin{equation} \label{eq:wesetC}
C = \sup_{z \in \R^d} \dfrac{\abs*{\nabla V(z)}}{4 \left(1 + V(z)\right)^{3/4}} < \infty ,
\end{equation}
which is well-defined under Assumption~\ref{assum:assumptionspotentialdampedeq}, and set $C_V := (2^{1/4} + C)^3 \ge 1$, which depends only on $V$. Fix $A_0 > 0$ such that $V(x) \ge 1$ whenever $\abs*{x} \ge A_0$ (it exists since $V$ is confining under~\eqref{eq:assumV}) and set
\begin{equation} \label{eq:defepsilonlambda}
\epsilon(\lambda) = C_V \inf_{A \ge A_0} \left( \dfrac{\norm*{\nabla V}_{L^\infty(B_A(0))}}{\lambda^{3/2}} + \sup_{\abs*{z} \ge A} \dfrac{\abs*{\nabla V(z)}}{V(z)^{3/4}} \right) .
\end{equation}
From this definition, it is clear that $\epsilon$ is well-defined for any $\lambda > 0$, takes positive values ($V$ cannot vanish identically since it is confining) and that it is indeed non-increasing. To show that $\epsilon(\lambda) \to 0$ as $\lambda \to + \infty$, we proceed as follows. Let $\eta > 0$. We fix $A \ge A_0$ large enough so that the second term in the right-hand side of~\eqref{eq:defepsilonlambda} is smaller than $\eta/2$ (recall $V$ satisfies $\nabla V(x) = o(V(x)^{3/4})$ as $x \to \infty$). Then for any $\lambda \ge \norm*{\nabla V}_{L^\infty(B_A(0))}^{2/3} (2/\eta)^{2/3}$, the first term is smaller than $\eta/2$, which gives $\epsilon(\lambda) \le \eta$, and $\eta$ is arbitrary.

Take $\lambda \ge 1$ and $x_0 \in \R^d$ such that $V(x_0) \le \lambda^2$ as in the statement. Since $x \mapsto (1 + V(x))^{1/4}$ is of class $\cont^1$ on $\R^d$, we can apply the mean value inequality to have
\begin{equation*}
\forall x \in \bar B_{\sqrt{\lambda}}(x_0), \qquad \abs*{\left(1 + V(x)\right)^{1/4} - \left(1 + V(x_0)\right)^{1/4}}
	\le C \abs*{x - x_0}
	\le C \sqrt{\lambda} ,
\end{equation*}
with $C$ the constant in~\eqref{eq:wesetC}.
Recalling that $\lambda \ge 1$, we deduce that for any $x \in \bar B_{\sqrt{\lambda}}(x_0)$, it holds
\begin{equation*}
\left(1 + V(x)\right)^{3/4}
	\le \left(\left(1 + V(x_0)\right)^{1/4} + C \sqrt{\lambda} \right)^3
	\le \left((2 \lambda^2)^{1/4} + C \sqrt{\lambda}\right)^3
	\le (2^{1/4} + C)^3 \lambda^{3/2}
	= C_V\lambda^{3/2} .
\end{equation*}
After using the mean value inequality on $V$, it allows us to obtain for any $x \in \bar B_{\sqrt{\lambda}}(x_0)$ and any $A \ge A_0$:
\begin{align*}
\abs*{V(x) - V(x_0)}
	&\le \abs*{x - x_0} \sup_{t \in [0, 1]} \abs*{\nabla V(x_0 + t (x - x_0))} \\
	&\le \abs*{x - x_0} \left( \lambda^{3/2} \dfrac{\norm*{\nabla V(y)}_{L^\infty(B_A(0))}}{\lambda^{3/2}} + \sup_{y \in B_{\sqrt{\lambda}}(x_0) \setminus B_A(0)} \dfrac{\abs*{\nabla V(y)}}{V(y)^{3/4}} \left( 1 + V(y)\right)^{3/4} \right) \\
	&\le C_V \abs*{x - x_0} \lambda^{3/2} \left( \dfrac{\norm*{\nabla V(y)}_{L^\infty(B_A(0))}}{\lambda^{3/2}} + \sup_{y \in \R^d\setminus B_A(0)} \dfrac{\abs*{\nabla V(y)}}{V(y)^{3/4}} \right) ,
\end{align*}
and the result follows by taking the infimum over $A \ge A_0$.
\end{proof}

\begin{lemma}[Construction of quasimodes in the potential regime] \label{lem:potentialquasimodes}
Let $P$ be the operator defined in~\eqref{eq:defP} with a potential $V$ subject to~\eqref{eq:assumV}. Let $k \in \cont_\comp^\infty(\R^d)$ be non-negative, supported in $B_1(0)$ with $\norm*{k}_{L^2} = 1$. There exists $C > 0$ and $M > 0$ such that for any $x_0 \in \R^d$ with $\abs{x_0} \ge M$, writing $\lambda = \sqrt{V(x_0)}$, for any $R \in [1, \lambda]$, the function $u$ defined by
\begin{equation*}
u(x)
	= r^{-d/2} k\left(\dfrac{x - x_0}{r}\right) ,
		\qquad x \in \R^d, r = \dfrac{R}{\sqrt{\lambda}} ,
\end{equation*}
satisfies
\begin{equation*}
\norm*{u}_{L^2}
	= 1
		\qquad \rm{and} \qquad
\norm*{(P - \lambda^2) u}_{L^2}
	\le C \lambda \left( \dfrac{1}{R^2} + R \epsilon(\lambda) \right) ,
\end{equation*}
Here $\epsilon$ is the function from Lemma~\ref{lem:lemV} associated with $V$ and the constant $C$ depends only on the dimension $d$ and the choice of $k$.
\end{lemma}

\begin{proof}
The claim that $\norm*{u}_{L^2} = 1$ follows from the fact that $\norm*{k}_{L^2} = 1$ and the operator $M_r : f \mapsto r^{-d/2} f(\frac{\bullet - x_0}{r})$ is unitary.

We estimate the Laplacian of $u$: we have
\begin{equation*}
\norm*{\Delta u}_{L^2}
	= \norm*{\dfrac{1}{r^2} M_r (\Delta k)}_{L^2}
	\le \dfrac{\lambda}{R^2} \norm*{\Delta k}_{L^2} .
\end{equation*}
We estimate the potential energy now: using Lemma~\ref{lem:lemV} (recall that $u$ is supported in $B_r(x_0) \subset B_{\sqrt{\lambda}}(x_0)$), we have
\begin{equation*}
\norm*{(V - \lambda^2) u}_{L^2}
	\le \lambda^{3/2} \epsilon(\lambda) \norm*{\abs{\bullet - x_0} u}_{L^2}
	= \lambda R \epsilon(\lambda) \norm*{M_r \tilde k}_{L^2}
	= \lambda R \epsilon(\lambda) \norm*{\tilde k}_{L^2} ,
\end{equation*}
where $\tilde k(x) = \abs{x - x_0} k(x - x_0)$. Lemma~\ref{lem:lemV} works here provided $\abs{x_0} \ge M$ is large enough so that $\lambda = V(x_0)^{1/2} \ge 1$ (recall that $V$ is confining). We obtain the desired estimate.
\end{proof}

Now we prove Theorem~\ref{thm:necessarycondition}.

\begin{proof}[Proof of Theorem~\ref{thm:necessarycondition}]
We show the converse statement. We know by Proposition~\ref{prop:aprioriNC} that~\eqref{eq:GCC} is necessary. Thus, we assume that~\eqref{eq:TPC} is violated and prove that uniform stability fails. In virtue of Proposition~\ref{prop:quasimodes}, it suffices to exhibit a family of normalized $L^2$ functions $(u_\lambda)_\lambda$ in $\dom P$ such that
\begin{equation*}
(P - \lambda^2) u_\lambda
	= o_{L^2}(\lambda)
		\qquad \rm{and} \qquad
\inp*{u_\lambda}{b u_\lambda}_{L^2}
	= o(1)
\end{equation*}
as $\lambda \to \infty$.

That~\eqref{eq:TPC} does not hold implies that
\begin{equation*}
\forall n \in \N , \qquad
	\fint_{B_{R_n/V(x_{n, k})^{1/4}}(x_{n, k})} b(y) \dd y
		\strongto{k \to \infty} 0 ,
			\qquad R_n := n+1 ,
\end{equation*}
along a sequence of points $x_{n, k} \to \infty$ in $\R^d$. For all $n \in \N$, we set $x_n := x_{n, k_n}$ where $k_n$ is chosen sufficiently large in such a way that
\begin{equation} \label{eq:choiceparamn}
\forall n \in \N , \qquad
	\fint_{B_{R_n/V(x_n)^{1/4}}(x_n)} b(y) \dd y
		\le 2^{-n}
			\quad \rm{and} \quad
		R_n \le \min \left\{\dfrac{1}{\sqrt{\epsilon(\sqrt{V(x_n)})}}, V(x_n)^{1/2} \right\} ,
\end{equation}
where $\epsilon$ is the function associated with $V$ introduced in Lemma~\ref{lem:lemV}. This is possible since $V$ is confining and $\epsilon(\lambda) \to 0$ as $\lambda \to + \infty$.
Moreover, defining $\lambda_n = V(x_n)^{1/2}$ we can apply Lemma~\ref{lem:potentialquasimodes} to obtain a sequence of $L^2$ normalized functions $(u_n)_{n \in \N}$ satisfying for any $n$ large enough
\begin{equation*}
\norm*{\left( P - \lambda_n^2 \right) u_n}_{L^2}
	\le C \lambda_n \left( \dfrac{1}{R_n^2} + R_n \epsilon(\lambda_n) \right)
	= \lambda_n O\left( \dfrac{1}{(n+1)^2} + \sqrt{\epsilon(\lambda_n)} \right) = o(\lambda_n) ,
\end{equation*}
since $\epsilon(\lambda) \to 0$ as $\lambda \to + \infty$. Therefore $(u_n)_{n \in \N}$ is a $o(\lambda_n)$-quasimode of  $P$. It remains to prove that it is asymptotically undamped in the sense of Item~\ref{it:asymptundamped} of Proposition~\ref{prop:quasimodes}. By definition of $u_n$ in Lemma~\ref{lem:potentialquasimodes}, we have
\begin{equation*}
0 \le \int_{\R^d} \abs*{u_n(x)}^2 b(y) \dd y
	\le \dfrac{\norm*{k}_{L^\infty}^2}{(R_n/V(x_n)^{1/4})^d} \int_{B_{R_n/V(x_n)^{1/4}}(x_n)} b(y) \dd y
	\le \abs*{B_1(0)} 2^{-n} \norm*{k}_{L^\infty}^2 ,
		\qquad \forall n \in \N ,
\end{equation*}
where we used~\eqref{eq:choiceparamn} in the last inequality.
This tends to zero as $n \to \infty$, and the proof is complete.
\end{proof}

As an application of Theorem~\ref{thm:necessarycondition}, we have the following.

\begin{proof}[Proof of Corollary~\ref{cor:characterizationregulardamping}]
We already know that $b$ bounded from below outside a compact set is a sufficient condition by Proposition~\ref{prop:aprioriSC}. Conversely, it is necessary that~\eqref{eq:TPC} holds in virtue of Theorem~\ref{thm:necessarycondition}, namely
\begin{equation} \label{eq:lowerlimittpc}
\exists c > 0, \exists R > 0 : \qquad \liminf_{x \to \infty} \fint_{B_{R/V(x)^{1/4}}(x)} b(y) \dd y \ge c .
\end{equation}
Yet since $b$ is uniformly continuous, there exists $\eps > 0$ such that for any $x, y \in \R^d$ such that $\abs*{x - y} \le \eps$, we have $\abs*{b(x) - b(y)} \le c/3$. Now by choosing a parameter $A > 0$ large enough, we can achieve two things: on the one hand, in view of the lower limit in~\eqref{eq:lowerlimittpc}, we can secure
\begin{equation*}
\forall \abs{x} \ge A , \qquad
	\fint_{B_{R/V(x)^{1/4}}(x)} b(y) \dd y
		\ge \dfrac{2}{3} c ;
\end{equation*}
on the other hand, since $V$ is a confining potential, we can ensure that $R/V(x)^{1/4} \le \eps$ for all $\abs{x} \ge A$.
In particular, for any such $x$, there exists $y \in B_{R/V(x)^{1/4}}(x)$ such that $b(y) \ge 2c/3$. We conclude that $b(x) \ge b(y) - c/3 \ge c/3$ since $\abs*{x - y} \le \eps$, which completes the proof.
\end{proof}

\section{Stabilization condition} \label{sec:stabilizationcondition}

In this section, we prove Proposition~\ref{prop:SCgeom}. A noteworthy consequence of Assumption~\ref{assum:assumptionspotentialdampedeq} is that trajectories of the Hamiltonian flow $(\varphi^t)_{t \in \R}$ are well-approximated around a given point by their tangent for small times, as shown below.

For the sequel, it is convenient to introduce new variables adapted to the typical scales of our problem described in~\eqref{eq:typicalscales}. Let $t \mapsto \varphi^t(x, \xi) = (x^t, \xi^t)$ be a trajectory of the Hamiltonian flow, i.e.\ satisfying the o.d.e.\ system~\eqref{eq:Hamilton}, and suppose it is contained in the energy shell $\{p = \lambda^2\}$.
We set $s = t \lambda$ and consider the reparametrization $s \mapsto (x^{s/\lambda}, \xi^{s/\lambda})$. In doing so, we observe that the fact that $\xi^{s/\lambda}$ is the time derivative of $x^{s/\lambda}$ (with respect to $s$) is not true anymore. To correct this issue, we set
\begin{equation} \label{eq:linearizedflow}
\phi_\lambda^s(x, \xi)
	:= (y_s^\lambda, \eta_s^\lambda)
	= \left(x^{s/\lambda}, \dfrac{1}{\lambda} \xi^{s/\lambda}\right) ,
\end{equation}
which in turn satisfies
\begin{equation} \label{eq:odelinearizedflow}
\begin{split}
\left\{
\begin{aligned}
\dot y_s^\lambda &= \eta_s^\lambda \\
\dot \eta_s^\lambda &= - \dfrac{1}{\lambda^2} \nabla V(y_s^\lambda)
\end{aligned}
\right. \qquad \rm{and} \qquad p(y_s^\lambda, \lambda \eta_s^\lambda) = V(y_s^\lambda) + \dfrac{\lambda^2}{2} \abs*{\eta_s^\lambda}^2 = \lambda^2 , \; \forall s \in \R .
\end{split}
\end{equation}

\begin{lemma} \label{lem:lemV2}
Suppose $V$ is subject to~\eqref{eq:assumV} and Assumption~\ref{assum:assumptionspotentialdampedeq}. Recall the function $\epsilon$ from Lemma~\ref{lem:lemV}. Let $T > 0, \lambda > 0$, $(x, \xi) \in \{p = \lambda^2\}$, and introduce $(y_s^\lambda, \eta_s^\lambda) = \phi_\lambda^s(x, \xi)$ defined in~\eqref{eq:linearizedflow}. Set $(y, \eta) := (y_0^\lambda, \eta_0^\lambda)$. Then
\begin{equation*}
\forall s \in [-T, T], \qquad
\abs*{\eta_s^\lambda - \eta}
	\le\dfrac{T}{\sqrt{\lambda}} \epsilon(\lambda) \quad \rm{and} \quad \abs*{y_s^\lambda - (y + s \eta)}
	\le\dfrac{T^2}{\sqrt{\lambda}} \epsilon(\lambda) .
\end{equation*}
\end{lemma}

\begin{proof}
Recalling the o.d.e.\ system~\eqref{eq:odelinearizedflow} satisfied by $s \mapsto (y_s^\lambda, \eta_s^\lambda)$, we have
\begin{equation} \label{eq:estimatelengthcurve}
\dfrac{\dd}{\dd s} \eta_s^\lambda
	= - \dfrac{1}{\lambda^2} \nabla V(y_s^\lambda) .
\end{equation}
Therefore, using the mean value inequality, we have for any $s \in [-T, T]$:
\begin{equation} \label{eq:differencecurveslinearized}
\abs*{\eta_s^\lambda - \eta}
	= \abs*{\int_0^s - \dfrac{1}{\lambda^2} \nabla V(y_\tau^\lambda) \dd \tau}
	\le \dfrac{T}{\lambda^2} \sup_{s \in [-T, T]} \abs*{\nabla V(y_s^\lambda)}
	= \dfrac{T}{\sqrt{\lambda}} \sup_{s \in [-T, T]} \dfrac{\abs*{\nabla V(y_s^\lambda)}}{\lambda^{3/2}} .
\end{equation}
Yet for any $A \ge A_0$, we have
\begin{equation*}
\sup_{s \in [-T, T]} \dfrac{\abs*{\nabla V(y_s^\lambda)}}{\lambda^{3/2}}
	\le \dfrac{\norm*{\nabla V}_{L^\infty(B_A(0))}}{\lambda^{3/2}} + \sup_{\substack{s \in [-T, T] \\ \abs*{y_s^\lambda} \ge A}} \dfrac{\abs*{\nabla V(y_s^\lambda)}}{\lambda^{3/2}}
	\le \dfrac{\norm*{\nabla V}_{L^\infty(B_A(0))}}{\lambda^{3/2}} + \sup_{\abs*{z} \ge A} \dfrac{\abs*{\nabla V(z)}}{V(z)^{3/4}} ,
\end{equation*}
where we used the inequality $V(y_s^\lambda) \le \lambda^2$, as a consequence of the fact that $p$ is preserved by the Hamiltonian flow, namely $p(y_s^\lambda, \lambda \eta_s^\lambda) = \lambda^2, \forall s \in \R$ with our new variables. Combining this estimate with~\eqref{eq:differencecurveslinearized} and taking the infimum over $A \ge A_0$ yield the sought estimate for $\abs*{\eta_s^\lambda - \eta}$. Subsequently, applying the mean value inequality on $y_s^\lambda - (y + s \eta)$, we obtain for any $s \in [-T, T]$:
\begin{equation*}
\abs*{y_s^\lambda - (y + s \eta)} \le T \sup_{s \in [-T, T]} \abs*{\eta_s^\lambda - \eta} \le \dfrac{T^2}{\sqrt{\lambda}} \epsilon(\lambda) ,
\end{equation*}
which finishes the proof.
\end{proof}


\begin{proof}[Proof of Proposition~\ref{prop:SCgeom}]
Assume~\eqref{eq:SCdyn} holds for a fixed $b \in L^\infty(\R^d)$ with constants $T, R > 0$. First we prove that~\eqref{eq:GCC} holds (in fact we show~\eqref{eq:GCCnear0}, see Remark~\ref{rmk:GCCnear0}). From~\eqref{eq:SCdyn}, there exist constants $c, \lambda_0 > 0$ such that, if we pick an arbitrary point $(x_0, \nu_0) \in S \R^d$, considering $\rho_\lambda = (x_0, \lambda \nu_0) \in T^\star \R^d$, we have
\begin{equation} \label{eq:dsclambda0larger}
\forall \lambda \ge \lambda_0 , \qquad
	\fint_{-T}^T \left( b \ast \kappa_{R/\sqrt{\tilde \lambda}} \right) \left( \varphi_{t/\tilde \lambda}(\rho_\lambda) \right) \dd t
		\ge c ,
\end{equation}
where $\tilde \lambda^2 := p(\rho_\lambda)$. Notice that $\tilde \lambda \sim \lambda$ as $\lambda \to + \infty$.
Yet using Lemma~\ref{lem:lemV2}, we have $(\pi \circ \varphi_{t/\tilde \lambda})(\rho_\lambda) = x_0 + t \nu_0 + o(T^2/\sqrt{\tilde \lambda})$ as $\lambda \to + \infty$ uniformly in $t \in [-T, T]$, so that
\begin{multline*}
\abs*{\fint_{-T}^T \left( b \ast \kappa_{R/\sqrt{\tilde \lambda}} \right) \left( x_0 + t \nu_0 \right) \dd t - \int_{-T}^T \left( b \ast \kappa_{R/\sqrt{\tilde \lambda}} \right) \left( \varphi_{t/\tilde \lambda}(\rho_\lambda) \right) \dd t} \\
	\le \norm*{b}_{L^\infty} \norm*{\kappa_{R/\sqrt{\tilde \lambda}}\left(\bullet + x_0 + t \nu_0 - (\pi \circ \varphi_{t/\tilde \lambda})(\rho_\lambda)\right) - \kappa_{R/\sqrt{\tilde \lambda}}}_{L^1} \\
	= \norm*{b}_{L^\infty} \norm*{\kappa\left(\bullet + o\left(\frac{T^2}{R}\right)\right) - \kappa}_{L^1} ,
\end{multline*}
using the triangle inequality and a change of variables. The right-hand side goes to zero as $\lambda \to + \infty$ (by dominated convergence for instance) since $R$ and $T$ are independent of $\lambda$. Combining this with~\eqref{eq:dsclambda0larger} yields
\begin{equation*}
\fint_{-T}^T \left( b \ast \kappa_{R/\sqrt{\tilde \lambda}} \right) \left( x_0 + t \nu_0 \right) \dd t
	\ge c + o(1)
\end{equation*}
as $\lambda \to \infty$ (or equivalently $\tilde \lambda \to \infty$). This proves that~\eqref{eq:GCC} holds.

Next we prove that~\eqref{eq:TPC} holds. To this aim, we consider an arbitrary sequence of points $\rho_n = (x_n, 0)$ where $x_n \to \infty$ as $n \to \infty$ and we write $\lambda_n = \sqrt{V(x_n)}$ (which goes to $+ \infty$ since $V$ is a confining potential). Write for simplicity $x_n = x_\lambda$, $\rho_n = \rho_\lambda$. Then we have from~\eqref{eq:SCdyn}:
\begin{equation} \label{eq:tpclambda0greater}
\exists c, \lambda_0 > 0 : \forall \lambda \ge \lambda_0 , \qquad
	\fint_{-T}^T \left( b \ast \kappa_{R/V(x_\lambda)^{1/4}} \right) \left( \varphi_{t/\lambda}(\rho_\lambda) \right) \dd t 
		\ge c .
\end{equation}
Yet using Lemma~\ref{lem:lemV2} once again, we have $(\pi \circ \varphi_{t/\lambda})(\rho_\lambda) = x_\lambda + o(T^2/\sqrt{\lambda})$ as $\lambda \to + \infty$ uniformly in $t \in [-T, T]$, so that
\begin{multline*}
\abs*{\fint_{B_{R/\sqrt{\tilde \lambda}}(x_\lambda)} b(y) \dd y - \fint_{-T}^T \left( b \ast \kappa_{R/\sqrt{\tilde \lambda}} \right) \left( \varphi_{t/\lambda}(\rho_\lambda) \right) \dd t} \\
	\le \norm*{b}_{L^\infty} \norm*{\kappa_{R/\sqrt{\tilde \lambda}} - \kappa_{R/\sqrt{\tilde \lambda}}\left( \bullet + (\pi \circ \varphi_{t/\lambda})(\rho_\lambda) - x_\lambda \right)}_{L^1}
	= \norm*{b}_{L^\infty} \norm*{\kappa - \kappa\left( \bullet + o\left(\frac{T^2}{R}\right) \right)}_{L^1} ,
\end{multline*}
which tends to zero again as $\lambda \to + \infty$. The parameters $T$ and $R$ are fixed, so combining this with~\eqref{eq:tpclambda0greater} yields
\begin{equation*}
\fint_{B_{R/\sqrt{\tilde \lambda}}(x_\lambda)} b(y) \dd y
	\ge c + o(1)
\end{equation*}
as $\lambda \to \infty$ (or equivalently $\tilde \lambda \to \infty$). This proves that~\eqref{eq:TPC} holds.

\medskip
Conversely, suppose~\eqref{eq:GCC} holds with constant $T > 0$ and~\eqref{eq:TPC} holds with constant $R > 0$, both having lower bound $c > 0$. Choose $\tilde T = 2 T$ and $\tilde R = 2 R$ so that
\begin{equation} \label{eq:conditionRT}
\left(\dfrac{R}{\tilde R}\right)^4 \le 1 - \left(\dfrac{T}{\tilde T}\right)^2 .
\end{equation}
By Lemma~\ref{lem:lemV2} again, we have
\begin{equation} \label{eq:approxforGCCandTPC}
\abs*{\fint_{- \tilde T}^{\tilde T} \left( b \ast \kappa_{\tilde R/\sqrt{\lambda}} \right) \left( x + t \xi/\lambda \right) \dd t - \fint_{- \tilde T}^{\tilde T} \left( b \ast \kappa_{\tilde R/\sqrt{\lambda}} \right) \left( \varphi_{t/\lambda}(x, \xi) \right) \dd t}
	\strongto{\lambda \to + \infty} 0 ,
\end{equation}
uniformly in $(x, \xi) \in \{p = \lambda^2\}$. Let $(x, \xi) \in \{p = \lambda^2\}$. We distinguish two regimes.
\begin{itemize}[label=\textbullet]
\item Assume first that $\abs*{\xi} \ge \lambda/2$. Then use~\eqref{eq:approxforGCCandTPC} and change variables in $t$ ($t \mapsto \frac{\lambda}{\abs{\xi}} t$) to obtain
\begin{align*}
\fint_{- \tilde T}^{\tilde T} \left( b \ast \kappa_{\tilde R/\sqrt{\lambda}} \right) \left( \varphi_{t/\lambda}(x, \xi) \right) \dd t
	&= \fint_{- \tilde T \abs*{\xi}/\lambda}^{\tilde T \abs*{\xi}/\lambda} \left( b \ast \kappa_{\tilde R/\sqrt{\lambda}} \right) \left( x + t\dfrac{\xi}{\abs*{\xi}} \right) \dd t + o(1) \\
	&\ge \dfrac{T \lambda}{\tilde T \abs*{\xi}} \fint_{- T}^T \left( b \ast \kappa_{\tilde R/\sqrt{\lambda}} \right) \left( x + t\dfrac{\xi}{\abs*{\xi}} \right) \dd t + o(1)
	\ge \dfrac{T}{\sqrt{2} \tilde T} c + o(1)
\end{align*}
as $\lambda \to + \infty$. The inequalities come from the lower bound on $\abs*{\xi}$, the fact that $\abs*{\xi} \le \sqrt{2} \lambda$ (the potential is non-negative) and the assumption that~\eqref{eq:GCC} holds.
\item Else, if $\abs*{\xi} \le \lambda/2$, writing $(x_t^\lambda, \xi_t^\lambda) = \varphi_{t/\lambda}(x, \xi)$, we know by Lemma~\ref{lem:lemV2} that $\xi_t^\lambda = \xi + o(\tilde T \sqrt{\lambda})$ as $\lambda \to + \infty$, so in particular, $\abs*{\xi_t^\lambda} \le \sqrt{2} \lambda/2$ for any $\lambda$ large enough. Therefore,
\begin{equation*}
V(x_t^\lambda)
	= \lambda^2 - \dfrac{\abs*{\xi_t^\lambda}^2}{2}
	\ge \lambda^2 - \left(\dfrac{1}{2}\right)^2 \lambda^2
	\ge \left(\dfrac{1}{2}\right)^4 \lambda^2 ,
\end{equation*}
which implies that
\begin{equation*}
\dfrac{V(x_\lambda)^{1/4}}{R}
	\ge \dfrac{\sqrt{\lambda}}{\tilde R}
\end{equation*}
(recall that $\tilde R = 2 R$), hence $B_{\tilde R/\sqrt{\lambda}}(x_t^\lambda) \supset B_{R/V(x_t^\lambda)^{1/4}}(x_t^\lambda)$. We deduce that
\begin{equation*}
\fint_{- \tilde T}^{\tilde T} \left( b \ast \kappa_{\tilde R/\sqrt{\lambda}} \right) \left( \varphi_{t/\lambda}(x, \xi) \right) \dd t
	\ge \fint_{- \tilde T}^{\tilde T} \dfrac{\abs*{B_{R/V(x_t^\lambda)^{1/4}}(x_t^\lambda)}}{\abs*{B_{\tilde R/\sqrt{\lambda}}(x_t^\lambda)}} \fint_{B_{R/V(x_t^\lambda)^{1/4}}(x_t^\lambda)} b(y) \dd y \dd t .
\end{equation*}
Yet since $V(x_t^\lambda) \le \lambda^2$, we have
\begin{equation*}
\dfrac{\abs*{B_{R/V(x_t^\lambda)^{1/4}}(x_t^\lambda)}}{\abs*{B_{\tilde R/\sqrt{\lambda}}(x_t^\lambda)}}
	= \dfrac{R^d \lambda^{d/2}}{\tilde R^d V(x_t^\lambda)^{d/4}}
	\ge \left(\dfrac{R}{\tilde R}\right)^d ,
\end{equation*}
so that using~\eqref{eq:TPC}, we finally have
\begin{equation*}
\fint_{- \tilde T}^{\tilde T} \left( b \ast \kappa_{\tilde R/\sqrt{\lambda}} \right) \left( (\pi \circ \varphi_{t/\lambda})(x, \xi) \right) \dd t
	\ge \left(\dfrac{R}{\tilde R}\right)^d c + o(1) ,
\end{equation*}
as $\lambda \to + \infty$.
\end{itemize}
This finishes the proof of Proposition~\ref{prop:SCgeom}.
\end{proof}

\appendix

\section{Pseudo-differential operators} \label{app:pseudo}

Let us recall some definitions and classical results that we use throughout this article. The material is taken from~\cite{Lerner:10,Zworski:book}.

\paragraph{\textbf{Quantization.}}

One may extend the quantization procedure~\eqref{eq:defquantization} to tempered distributions on phase space: given $a \in \sch(T^\star \R^d)$ and a couple of test functions $u, v \in \sch(\R^d)$, we have the weak formulation of~\eqref{eq:defquantization}:
\begin{equation*}
\int_{\R^d} v(x) \Opw{a} u(x) \dd x
= \int_{\R^d \times \R^d} a(x, \xi) W[u, v](x, \xi) \dd x \dd \xi ,
\end{equation*}
where the \emph{Wigner transform} of $u$ and $v$ is defined by
\begin{equation*}
W[u, v](x, \xi) = (2 \pi)^{-d} \int_{\R^d} v\left( x + \dfrac{y}{2} \right) u \left( x - \dfrac{y}{2} \right) \e^{\ii \xi \cdot y} \dd y .
\end{equation*}
The Wigner transform is a continuous map $\sch(\R^d) \times \sch(\R^d) \to \sch(\R^{2d})$, which allows to make sense of $\Opw{a} : \sch(\R^d) \to \sch'(\R^d)$ for a tempered distribution $a \in \sch'(T^\star \R^d)$.

\paragraph{\textbf{Symbol classes, $L^2$ boundedness and pseudo-differential calculus.}}

We briefly recall classical definitions about order functions and symbol classes.

\begin{definition}[Order function]
A measurable function $m : T^\star \R^d \to R_+^\star$ is called an order function is there exist constants $C > 0$ and $N \in R$ such that
\begin{equation*}
m(\rho_2) \le C \jap{\rho_2 - \rho_1}^N m(\rho_1) ,
	\qquad \forall \rho_1, \rho_2 \in T^\star \R^d .
\end{equation*}
\end{definition}


\begin{definition}[Symbol class $S(m)$]
Let $m$ be an order function. The symbol class $S(m)$ is
\begin{equation*}
S(m) = \set{a \in \cont^\infty(T^\star \R^d)}{\forall \alpha \in \N^{2d}, \exists C_\alpha > 0 : \abs*{\partial^\alpha a} \le C_\alpha m} .
\end{equation*}
\end{definition}

For any multi-index $\alpha \in \N^{2d}$, the best constant $C_\alpha$ provides a seminorm, which endows $S(m)$ with a structure of Fréchet space. Note that $\partial^\alpha : S(m) \to S(m)$ is continuous, and that $(a_1, a_2) \mapsto a_1 a_2$ maps $S(m_1) \times S(m_2) \to S(m_1 m_2)$ continuously.
%
The symbols in $S(1)$ (or in $S(m)$ with $m$ bounded) enjoy a good boundedness property on $L^2$:

\begin{theorem}[{Calder\'{o}n--Vaillancourt \--- \cite[Theorem 4.23]{Zworski:book}}] \label{thm:CV}
Let $a \in S(1)$. Then $\Opw{a}$ can be uniquely extended to a bounded operator $L^2(\R^d) \to L^2(\R^d)$ with the estimate
\begin{equation*}
\norm*{\Opw{a}}_{\Bop\left(L^2(\R^d)\right)} \le C_d \sum_{\abs*{\alpha} \le M_d} \norm*{\partial^\alpha a}_\infty ,
\end{equation*}
where the constants $C_d$ and $M_d$ depend only on the dimension $d$.
\end{theorem}

In the usual semiclassical Weyl quantization setting~\eqref{eq:semiclassicalrescalingLap}, composition of pseudo-differential operators behaves as follows.

\begin{theorem}[{Pseudo-differential calculus \--- \cite[Theorem 4.18]{Zworski:book}}] \label{thm:pseudodiffcalc}
Let $m_1, m_2$ be order functions, and let $a_1 \in S(m_1), a_2 \in S(m_2)$. Then we have
\begin{align*}
\Ophw{a_1} \Ophw{a_2}
	&= \Ophw{a_1 a_2} + h \Ophw{r_1(a_1, a_2; h)} \\
	&= \Ophw{a_1 a_2} - \dfrac{\ii h}{2} \Ophw{\{a_1, a_2\}} + h^2 \Ophw{r_2(a_1, a_2; h)} ,
\end{align*}
as operators on $\sch(\R^d)$, where $r_1(\bullet, \bullet; h)$ and $r_2(\bullet, \bullet; h)$ are bilinear maps from $S(m_1) \times S(m_2)$ to $S(m_1 m_2)$ that are continuous uniformly in $h \in (0, 1]$. Here $\poiss{a_1}{a_2} = \partial_\xi a_1 \cdot \partial_x a_2 - \partial_\xi \cdot a_2 \cdot \partial_x a_1$ is the Poisson bracket. In particular, if $m_1 m_2$ is bounded, we have
\begin{align*}
\Ophw{a_1} \Ophw{a_2}
	&= \Ophw{a_1 a_2} + O_{\Bop\left(L^2(\R^d)\right)}(h) \\
	&= \Ophw{a_1 a_2} - \dfrac{\ii h}{2} \Ophw{\{a_1, a_2\}} + O_{\Bop\left(L^2(\R^d)\right)}(h^2) ,
\end{align*}
as operators on $L^2(\R^d)$.
\end{theorem}

\section{Miscellaneous results on classical averages} \label{app:constant}

We start this section with a technical remark about~\eqref{eq:GCC}.

\begin{lemma} \label{lem:GCCnear0}
Let $b \in L^\infty(\R^d)$ be non-negative and $\kappa_r$ be the convolution kernel introduced in~\eqref{eq:defkappa}. Then the condition \eqref{eq:GCC} is equivalent to the following statement:
\begin{equation} \label{eq:GCCnear0-}
\exists T > 0, \exists c > 0 : \forall (x_0, \nu_0) \in S \R^d, \qquad \liminf_{r \to 0} \fint_{-T}^T (b \ast \kappa_r)(x_0 + t \nu_0) \dd t \ge c .
\end{equation}
\end{lemma}

\begin{proof}
Condition \eqref{eq:GCC} clearly implies the property~\eqref{eq:GCCnear0-}. For the converse, fix $r_0 > 0$, $(x_0, \nu_0) \in S \R^d$ and assume~\eqref{eq:GCCnear0-} holds. Using the associative and commutative properties of the convolution together with Fubini's Theorem, for any $r > 0$ we have
\begin{align*}
\fint_{-T}^T \left( \kappa_r \ast (b \ast \kappa_{r_0}) \right)(x_0 + t \nu_0) \dd t
	&= \fint_{-T}^T \left( \kappa_{r_0} \ast (b \ast \kappa_r) \right)(x_0 + t \nu_0) \dd t \\
	&= \int_{\R^d} \kappa_{r_0}(y) \fint_{-T}^T (b \ast \kappa_r) (x_0 + t \nu_0 - y) \dd t \dd y .
\end{align*}
Now letting $r \to 0$, we observe that the left-hand side tends to $\fint_{-T}^T (b \ast \kappa_{r_0}) (x_0 + t \nu_0) \dd t$ (by the dominated convergence theorem for example, using that $b \ast \kappa_{r_0}$ is continuous). As for the right-hand side, Fatou's Lemma yields
\begin{align*}
\liminf_{r \to 0} \int_{\R^d} \kappa_{r_0}(y) \fint_{-T}^T (b \ast \kappa_r) (x_0 + t \nu_0 - y) \dd t \dd y
	&\ge \int_{\R^d} \kappa_{r_0}(y) \liminf_{r \to 0} \fint_{-T}^T (b \ast \kappa_r) (x_0 - y + t \nu_0) \dd t \dd y \\
	&\ge c \norm*{\kappa}_{L^1} ,
\end{align*}
where we used~\eqref{eq:GCCnear0-} for the last inequality. Therefore~\eqref{eq:GCC} is proved.
\end{proof}

Notice that when the damping coefficient $b$ is continuous, \eqref{eq:GCCnear0-} proves that~\eqref{eq:GCC} is equivalent to
\begin{equation*}
\exists T > 0, \exists c > 0 : \forall (x_0, \nu_0) \in S \R^d, \qquad  \fint_{-T}^T b(x_0 + \tau \nu_0) \dd \tau \ge c .
\end{equation*}

We finally suggest a way to quantify the fact that a function $b$ satisfies~\eqref{eq:SCdyn}. The relevant quantity seems to be the limit of~\eqref{eq:SCdyn} as $(T, R) \to \infty$, which turns out to exist.

\begin{proposition} \label{prop:limitSC}
Suppose $V$ is subject to~\eqref{eq:assumV} and Assumption~\ref{assum:assumptionspotentialdampedeq}. For any $b \in L^\infty(\R^d)$ non-negative, the limit
\begin{equation} \label{eq:SCdyninfty}
\lim_{(T, R) \to \infty} \liminf_{\lambda \to + \infty} \inf_{\{p = \lambda^2\}} \avg*{b \ast \kappa_{R/\sqrt{\lambda}}}_{T/\lambda}
\end{equation}
exists. Moreover, $b$ satisfies~\eqref{eq:SCdyn} if and only if this limit is positive.
\end{proposition}

\begin{proof}
Let $(T, R)$ and $(\tilde T, \tilde R)$ be two couples of positive numbers. First, for any $\lambda \ge 1$ and any $\rho \in \{p = \lambda^2\}$, we have by Fubini's theorem
\begin{multline*}
\abs*{\avg*{\avg*{b \ast \kappa_{\tilde R/\sqrt{\lambda}} \ast \kappa_{R/\sqrt{\lambda}}}_{T/\lambda}}_{\tilde T/\lambda}(\rho) - \avg*{b \ast \kappa_{\tilde R/\sqrt{\lambda}} \ast \kappa_{R/\sqrt{\lambda}}}_{\tilde T/\lambda}(\rho)} \\
	= \abs*{\fint_{-T}^{T} \left( \fint_{-\tilde T + s}^{\tilde T + s} \left(b \ast \kappa_{\tilde R/\sqrt{\lambda}} \ast \kappa_{R/\sqrt{\lambda}}\right)\left(\varphi_{\tau/\lambda}(\rho)\right) \dd \tau - \fint_{-\tilde T}^{\tilde T} \left(b \ast \kappa_{\tilde R/\sqrt{\lambda}} \ast \kappa_{R/\sqrt{\lambda}}\right)\left(\varphi_{\tau/\lambda}(\rho)\right) \dd \tau \right) \dd s} \\
	= \abs*{\fint_{-T}^{T} \left( \fint_{-\tilde T + s}^{\tilde T + s} - \fint_{-\tilde T}^{\tilde T}\right) \left(b \ast \kappa_{\tilde R/\sqrt{\lambda}} \ast \kappa_{R/\sqrt{\lambda}}\right)\left(\varphi_{\tau/\lambda}(\rho)\right) \dd \tau \dd s} \\
	\le \fint_{-T}^T 2 \dfrac{\abs*{s}}{\tilde T} \norm*{b \ast \kappa_{\tilde R/\sqrt{\lambda}} \ast \kappa_{R/\sqrt{\lambda}}}_{L^\infty} \dd s
	\le 2 \dfrac{T}{\tilde T} \norm*{b}_{L^\infty} \norm*{\kappa}_{L^1}^2 .
\end{multline*}
Second, we have for all $\rho \in T^\star \R^d$~:
\begin{align*}
\abs*{\avg*{b \ast \kappa_{\tilde R/\sqrt{\lambda}} \ast \kappa_{R/\sqrt{\lambda}}}_{\tilde T/\lambda}(\rho) - \avg*{b \ast \kappa_{\tilde R/\sqrt{\lambda}}}_{\tilde T/\lambda}(\rho)}
	&\le \norm*{b \ast \kappa_{\tilde R/\sqrt{\lambda}} \ast \kappa_{R/\sqrt{\lambda}} - b \ast \kappa_{\tilde R/\sqrt{\lambda}}}_{L^\infty} \\
	&\hspace*{-1.5cm}\le \int_{\R^d} \kappa_{R/\sqrt{\lambda}}(y) \norm*{\left(b \ast \kappa_{\tilde R/\sqrt{\lambda}}\right) (\bullet - y) - b \ast \kappa_{\tilde R/\sqrt{\lambda}}}_{L^\infty} \dd y \\
	&\hspace*{-1.5cm}\le \norm*{\kappa}_{L^1} \norm*{b}_{L^\infty} \sup_{y \in B_{R/\sqrt{\lambda}}(0)} \norm*{\kappa_{\tilde R/\sqrt{\lambda}}(\bullet - y) - \kappa_{\tilde R/\sqrt{\lambda}}}_{L^1} \\
	&\hspace*{-1.5cm}= \norm*{\kappa}_{L^1} \norm*{b}_{L^\infty} \sup_{y \in B_1(0)} \norm*{\kappa\left(\bullet - y\frac{R}{\tilde R}\right) - \kappa}_{L^1} .
\end{align*}
In the first inequality, we use a crude $L^\infty$ bound to get rid of the time average $\avg{\bullet}_{\tilde T / \lambda}$. Then we expand the convolution with $\kappa_{R/\sqrt{\lambda}}$ and use Young's inequality in $L^1$--$L^\infty$. Then we apply Young's inequality again to have the third inequality (this time the $L^\infty$ norm falls on $b$ and the $L^1$ norm on $\kappa_{\tilde R/\sqrt{\lambda}}$). The last equality consists in a change of variables.
Therefore combining the two estimates, we obtain
\begin{equation} \label{eq:relationRTtilde}
\avg*{\avg*{b \ast \kappa_{\tilde R/\sqrt{\lambda}} \ast \kappa_{R/\sqrt{\lambda}}}_{\tilde T/\lambda}}_{T/\lambda}(\rho)
	= \avg*{b \ast \kappa_{\tilde R/\sqrt{\lambda}}}_{\tilde T/\lambda}(\rho) + o(1)
\end{equation}
as $(\tilde T, \tilde R) \to \infty$, uniformly in $\lambda \ge 1$ and $\rho \in \{p = \lambda^2\}$. Now, $\tilde R$ being fixed, we want to take the infimum of~\eqref{eq:relationRTtilde} over $\rho \in \{p = \lambda^2\}$. For the left-hand side, there is some additional work. Let $(x_0, \xi_0) \in \{p = \lambda^2\}$ and $y \in B_{\tilde R/\sqrt{\lambda}}(0)$, depending on $\lambda$. Using Lemma~\ref{lem:lemV}, we know that $\abs*{V(x_0 + y) - V(x_0)} = o(\tilde R \lambda)$ as $\lambda \to + \infty$ so that
\begin{equation} \label{eq:equivalencelambdatilde}
p(x_0 + y, \xi_0)
	\sim \lambda^2 ,
		\qquad \rm{as} \; \lambda \to + \infty ,
\end{equation}
and by Lemma~\ref{lem:lemV2}, we know that
\begin{equation} \label{eq:linapprx}
(\pi \circ \varphi_{s/\lambda})(x_0, \xi_0) - y
	= x_0 - y + \dfrac{s}{\lambda} \xi_0 + o\left( \dfrac{T^2}{\sqrt{\lambda}} \right)
	= (\pi \circ \varphi_{s/\lambda})(x_0 - y, \xi_0) + o\left( \dfrac{T^2}{\sqrt{\lambda}} \right)
\end{equation}
as $\lambda \to + \infty$, uniformly in $s \in [-T, T]$, $(x_0, \xi_0) \in \{p = \lambda^2\}$ and $y \in B_{\tilde R/\sqrt{\lambda}}(0)$. We expand the convolution with $\kappa_{\tilde R/\sqrt{\lambda}}$
\begin{align}
I
	&:=\fint_{- \tilde T}^{\tilde T} \fint_{-T}^T \left( b \ast \kappa_{\tilde R/\sqrt{\lambda}} \ast \kappa_{R/\sqrt{\lambda}} \right) \Bigl( \left( \pi \circ \varphi_{(t + s)/\lambda} \right) (\rho) \Bigr) \dd s \dd t \label{eq:defI:=}\\
	&= \fint_{- \tilde T}^{\tilde T} \int_{\R^d} \kappa_{\tilde R /\sqrt{\lambda}}(y) \fint_{-T}^T \left( b \ast \kappa_{R/\sqrt{\lambda}} \right) \Bigl( \left( \pi \circ \varphi_{s/\lambda}\right) \left(\varphi_{t/\lambda}(\rho)\right) - y \Bigr) \dd s \dd y \dd t . \nonumber
\end{align}
Writing for short $x' = \pi \circ \varphi_{s/\lambda}(\varphi_{t/\lambda}(\rho)) - y$ and $\tilde x = \pi \circ \varphi_{s/\lambda} (\varphi_{t /\lambda}(\rho) - (y, 0))$, we have
\begin{multline}
\abs*{I - \fint_{- \tilde T}^{\tilde T} \int_{\R^d} \kappa_{\tilde R /\sqrt{\lambda}}(y) \fint_{-T}^T \left( b \ast \kappa_{R/\sqrt{\lambda}} \right) \left( \varphi_{s/\lambda} \left(\varphi_{t/\lambda}(\rho) - (y, 0)\right) \right) \dd s \dd y \dd t} \nonumber\\
	\le \norm*{\kappa}_{L^1} \sup_{\substack{t \in [-\tilde T, \tilde T] \\ s \in [-T, T]}} \sup_{\substack{\rho \in \{p = \lambda^2\} \\ y \in B_{\tilde R/\sqrt{\lambda}}(0)}} \abs*{\left( b \ast \kappa_{R/\sqrt{\lambda}} \right)\left(x'\right) - \left( b \ast \kappa_{R/\sqrt{\lambda}} \right)\left(\tilde x\right)} .
\end{multline}
We apply~\eqref{eq:linapprx} to $(x_0, \xi_0) = \varphi_{t/\lambda}(\rho)$ to deduce that $\abs{x' - \tilde x} = o(T^2/\sqrt{\lambda})$, so that
\begin{align*}
\abs*{\left( b \ast \kappa_{R/\sqrt{\lambda}} \right)\left(x'\right) - \left( b \ast \kappa_{R/\sqrt{\lambda}} \right)\left(\tilde x\right)}
	&= \abs*{\int_{\R^d} \left(b(x' - y') \kappa_{R/\sqrt{\lambda}}(y') - b(x' - y') \kappa_{R/\sqrt{\lambda}}(y' + \tilde x - x') \right) \dd y'} \\
	&\le \norm*{b}_{L^\infty} \norm*{\kappa - \kappa\left( \bullet + \dfrac{\tilde x - x'}{R/\sqrt{\lambda}}\right)}_{L^1}
	\strongto{\lambda \to + \infty} 0 ,
\end{align*}
by dominated convergence. This is uniform in $t, s, \rho$ and $y$ (recall that $x'$ and $\tilde x$ depend on all these variables). So we deduce that
\begin{equation} \label{eq:I+o(1)}
I
	= \fint_{- \tilde T}^{\tilde T} \int_{\R^d} \kappa_{\tilde R /\sqrt{\lambda}}(y) \fint_{-T}^T \left( b \ast \kappa_{R/\sqrt{\lambda}} \right) \left( \varphi_{s/\lambda} \left(\varphi_{t/\lambda}(\rho) - (y, 0)\right) \right) \dd s \dd y \dd t + o(1)
\end{equation}
as $\lambda \to + \infty$, uniformly in $\rho \in \{p = \lambda^2\}$.
Now we write $\tilde \rho = \varphi_{t/\lambda}(\rho) - (y, 0)$ and we denote by $\tilde \lambda$ the real number such that $\tilde \lambda^2 = p(\tilde \rho)$. Bear in mind that $\tilde \lambda \sim \lambda$ as $\lambda \to + \infty$ from~\eqref{eq:equivalencelambdatilde}. Now by the triangle inequality, we have
\begin{align*}
D
	:= \Biggl| \fint_{-T}^T \bigl( b \,\ast & \,\kappa_{R/\sqrt{\lambda}} \bigr) \left( \varphi_{s/\lambda} (\tilde \rho)\right) \dd s - \fint_{-T}^T \left( b \ast \kappa_{R/\sqrt{\tilde \lambda}} \right) \left( \varphi_{s/\tilde \lambda} (\tilde \rho) \right) \dd s \Biggr| \\
	&\le \abs*{\fint_{-T}^T \left( b \ast \kappa_{R/\sqrt{\lambda}} \right) \left( \varphi_{s/\lambda} (\tilde \rho)\right) \dd s - \fint_{-T}^T \left( b \ast \kappa_{R/\sqrt{\tilde \lambda}} \right) \left( \varphi_{s/\lambda} (\tilde \rho) \right) \dd s} \\
	&\qquad\qquad	+ \abs*{\fint_{-T}^T \left( b \ast \kappa_{R/\sqrt{\tilde \lambda}} \right) \left( \varphi_{s/\lambda} (\tilde \rho)\right) \dd s - \fint_{-T}^T \left( b \ast \kappa_{R/\sqrt{\tilde \lambda}} \right) \left( \varphi_{s/\tilde \lambda} (\tilde \rho) \right) \dd s} .
\end{align*}
Changing variables in the integral over $s$ and using Young's inequality, we obtain
\begin{align*}
D
	&\le \norm*{b}_{L^\infty} \norm*{\kappa_{R/\sqrt{\lambda}} - \kappa_{R/\sqrt{\tilde \lambda}}}_{L^1}
		+ \abs*{\fint_{-T}^T \left( b \ast \kappa_{R/\sqrt{\tilde \lambda}} \right) \left( \varphi_{s/\lambda} (\tilde \rho)\right) \dd s - \fint_{-T \frac{\lambda}{\tilde \lambda}}^{T \frac{\lambda}{\tilde \lambda}} \left( b \ast \kappa_{R/\sqrt{\tilde \lambda}} \right) \left( \varphi_{s/\lambda} (\tilde \rho) \right) \dd s} \\
	&\le \norm*{b}_{L^\infty} \norm*{\kappa_{\sqrt{\tilde \lambda/\lambda}} - \kappa}_{L^1} + \norm*{b}_{L^\infty} \norm*{\kappa}_{L^1} \left( \dfrac{1}{2 T \frac{\lambda}{\tilde \lambda}} \times 2 T \abs*{\frac{\lambda}{\tilde \lambda} - 1} + 2 T \abs*{\dfrac{1}{2T} - \dfrac{1}{2T \frac{\lambda}{\tilde \lambda}}} \right) ,
\end{align*}
which tends to zero as $\lambda \to + \infty$ (recall~\eqref{eq:equivalencelambdatilde}). Therefore plugging $D \to 0$ into~\eqref{eq:I+o(1)}, we deduce that
\begin{equation*}
I
	= \fint_{- \tilde T}^{\tilde T} \int_{\R^d} \kappa_{\tilde R /\sqrt{\lambda}}(y) \avg*{b \ast \kappa_{R/\sqrt{\tilde \lambda}}}_{T/\tilde \lambda}(\tilde \rho) \dd y \dd t + o(1) ,
\end{equation*}
hence recalling the definition of $I$ in~\eqref{eq:defI:=}:
\begin{align*}
I
	&= \fint_{- \tilde T}^{\tilde T} \fint_{-T}^T \left( b \ast \kappa_{\tilde R/\sqrt{\lambda}} \ast \kappa_{R/\sqrt{\lambda}} \right) \left( \left( \pi \circ \varphi_{(t + s)/\lambda} \right) (\rho) \right) \dd s \dd t \\
	&\ge \inf_{\tilde \lambda \ge \lambda/2} \inf_{\tilde \rho \in \{p = \tilde \lambda^2\}} \avg*{b \ast \kappa_{R/\sqrt{\tilde \lambda}}}_{T/\tilde \lambda}(\tilde \rho) \norm*{\kappa}_{L^1} + o(1)
\end{align*}
as $\lambda \to + \infty$. Taking the infimum over $\rho \in \{p = \lambda^2\}$ in the left-hand side yields
\begin{equation*}
\inf_{\rho \in \{p = \lambda^2\}} \avg*{\avg*{b \ast \kappa_{\tilde R/\sqrt{\lambda}} \ast \kappa_{R/\sqrt{\lambda}}}_{\tilde T/\lambda}}_{T/\lambda}(\rho)
	\ge \inf_{\tilde \lambda \ge \lambda/2} \inf_{\tilde \rho \in \{p = \tilde \lambda^2\}} \avg*{b \ast \kappa_{R/\sqrt{\tilde \lambda}}}_{T/\tilde \lambda}(\tilde \rho) + o(1)
\end{equation*}
and taking lower limits in $\lambda$ together with~\eqref{eq:relationRTtilde} imply:
\begin{equation*}
\liminf_{(\tilde T, \tilde R) \to \infty} \liminf_{\lambda \to +\infty} \inf_{\{p = \lambda^2\}} \avg*{b \ast \kappa_{\tilde R/\sqrt{\lambda}}}_{\tilde T/\lambda}
	\ge \liminf_{\lambda \to + \infty} \inf_{\{p = \lambda^2\}} \avg*{b \ast \kappa_{R/\sqrt{\lambda}}}_{T/\lambda} .
\end{equation*}
This estimate proves that if $b$ satisfies~\eqref{eq:SCdyn} with constants $(T, R)$, then the lower limit in $(\tilde T, \tilde R)$ above is positive, and the converse is straightforward. Taking the upper limit as $(T, R) \to \infty$ of the right-hand side provides the existence of the limit in~\eqref{eq:SCdyninfty}.
\end{proof}



\small
\bibliographystyle{alpha}
\bibliography{biblio}
\end{document}

%% file: macros_Antoine.tex
\newtheorem{lemma}{Lemma}[section]
\newtheorem{theorem}[lemma]{Theorem}
\newtheorem{proposition}[lemma]{Proposition}
\newtheorem{corollary}[lemma]{Corollary}
\newtheorem{assumption}[lemma]{Assumption}
\newtheorem{conjecture}[lemma]{Conjecture}

\theoremstyle{definition}

\newtheorem{definition}[lemma]{Definition}
\newtheorem{remark}[lemma]{Remark}

\newcommand*{\N}{\mathbf{N}}

\newcommand*{\R}{\mathbf{R}}
\newcommand*{\CC}{\mathbf{C}}

\newcommand*{\eps}{\varepsilon}

\newcommand*{\cal}[1]{{\mathcal{#1}}}
\newcommand*{\scr}[1]{{\mathscr{#1}}}
\renewcommand*{\frak}[1]{{\mathfrak{#1}}}
\renewcommand*{\rm}[1]{{\mathrm{#1}}}

\newcommand*{\one}{\mathds{1}}

\newcommand*{\sph}{S}

\newcommand*{\e}{e}
\newcommand*{\ii}{i}
\newcommand*{\dd}{\mathop{}\mathopen{}{d}}

\newcommand*{\cont}{C}

\newcommand*{\sch}{{\mathscr S}}
\DeclareMathOperator{\ft}{{\mathscr F}}

\DeclarePairedDelimiterX{\brak}[2]{\langle}{\rangle}{#1, #2}
\DeclarePairedDelimiterX{\inp}[2]{(}{)}{#1, #2}
\DeclarePairedDelimiterX{\comm}[2]{[}{]}{#1, #2}
\DeclarePairedDelimiterX{\poiss}[2]{\{}{\}}{#1, #2}
\DeclarePairedDelimiterX{\liebrak}[2]{[}{]}{#1, #2}
\DeclarePairedDelimiter{\parens}{(}{)}
\DeclarePairedDelimiter\abs{\lvert}{\rvert}
\DeclarePairedDelimiter\norm{\lVert}{\rVert}
\DeclarePairedDelimiter\jap{\langle}{\rangle}

\newcommand*{\comp}{c}
\newcommand*{\loc}{\mathrm{loc}}

\DeclareMathOperator{\dom}{dom}

\DeclareMathOperator{\essinf}{ess inf}

\DeclareMathOperator{\Span}{span}
\DeclareMathOperator{\supp}{supp}
\DeclareMathOperator{\dist}{dist}

\let \Re \relax
\DeclareMathOperator{\Re}{Re}
\let \Im \relax
\DeclareMathOperator{\Im}{Im}
\newcommand*{\Lap}{\Delta} 

\DeclareMathOperator{\quantization}{Op}

\newcommand*{\Opw}[2][]{{\quantization}_{#1}^{\mathrm{\scriptscriptstyle W}}\hspace{-0.075em}\parens*{#2}}
\newcommand*{\Ophw}[1]{{\quantization}_h^{\mathrm{\scriptscriptstyle W}}\hspace{-0.1em}\parens*{#1}}

\DeclareMathOperator{\id}{Id}

\newcommand*{\strongto}[2][]{\xrightarrow[#2]{#1}}

\renewcommand*{\set}[2]{\left\{#1:#2\right\}}

\newlength\oversetwidth
\newlength\underwidth

